\documentclass[12pt]{amsart}
\usepackage{etex}

\usepackage{amsthm,amssymb,enumitem,mathtools}

\usepackage{cite}

\usepackage{tikz}

\usepackage{scalerel}

\usepackage{cancel}

\usepackage{framed}

\usetikzlibrary{arrows,shapes,automata,backgrounds,decorations,petri,positioning}
\usetikzlibrary[calc,intersections,through,backgrounds,arrows,decorations.pathmorphing,decorations.markings]

\tikzstyle{edge} = [fill,opacity=.5,fill opacity=.5,line cap=round, line join=round, line width=50pt]

\pgfdeclarelayer{background}
\pgfsetlayers{background,main}

\usepackage[margin=1in,letterpaper,portrait]{geometry}

\usepackage{amsthm}

\usepackage[latin1]{inputenc}
\usepackage{subfigure}
\usepackage{color}
\usepackage{amsmath}
\usepackage{amsthm}
\usepackage{amstext}
\usepackage{amssymb}
\usepackage{amsfonts}
\usepackage{graphicx}
\usepackage{young}
\usepackage{multicol}
\usepackage{mathrsfs}
\usepackage[all]{xy}
\usepackage{tikz}
\usepackage{mathtools}
\usepackage{tabularx}
\usepackage{array}
\usepackage{commath}
\usepackage{ytableau}

\usepackage{scrextend}

\usepackage{hyperref}

\theoremstyle{plain}
\theoremstyle{definition}

\newtheorem{theorem}{Theorem}[section]

\newtheorem{remark}[theorem]{Remark}
\newtheorem{lemma}[theorem]{Lemma}

\newtheorem{definition}[theorem]{Definition}

\newtheorem{example}[theorem]{Example}
\newtheorem{proposition}[theorem]{Proposition}
\newtheorem{corollary}[theorem]{Corollary}

\DeclareMathAlphabet{\mathpzc}{OT1}{pzc}{m}{it}

\newcommand{\s}{\sigma}

\newcommand{\supp}{\textnormal{\textsf{supp}}}
\newcommand{\mf}[1]{\mbox{$\mathfrak #1$}}
\newcommand{\rw}[1]{\left[{#1}\right]}
\newcommand{\run}{\textnormal{\textsf{run}}}
\newcommand{\addletter}{\curlywedge}
\newcommand{\optimalrank}{\textsf{ork}}
\newcommand{\shape}{\lambda}
\newcommand{\row}{\textsf{row}}
\newcommand{\RowBT}{\textsf{Row}}

\begin{document}

\title[Intersecting principal Bruhat ideals]{Intersecting principal Bruhat ideals\\
and grades of simple modules}

\date{}

\author[Volodymyr Mazorchuk]{Volodymyr Mazorchuk$^{\dagger}$}
\address{Department of Mathematics, Uppsala University, Uppsala, SWEDEN}
\email{mazor@math.uu.se}
\thanks{$^{\dagger}$Research partially supported by the Swedish Research Council and G{\"o}ran Gustafsson Stiftelse}

\author[Bridget Eileen Tenner]{Bridget Eileen Tenner$^*$}
\address{Department of Mathematical Sciences, DePaul University, Chicago, IL, USA}
\email{bridget@math.depaul.edu}
\thanks{$^*$Research partially supported by NSF Grant DMS-2054436 and Simons Foundation Collaboration Grant for Mathematicians 277603.}



\begin{abstract}%
We prove that the grades of simple modules indexed by boolean permutations, 
over the incidence algebra of the symmetric group with respect to the Bruhat order, 
are given by Lusztig's $\mathbf{a}$-function. Our arguments are combinatorial, and include a description of the intersection of two principal order ideals when at least one permutation is boolean. An important object in our work is a reduced word written as minimally many runs of consecutive integers, and one step of our argument shows that this minimal quantity is equal to the length of the second row in the permutation's shape under the Robinson-Schensted correspondence. We also prove that 
a simple module over the above-mentioned incidence algebra is perfect if and only if its
index is the longest element of a parabolic subgroup.
\end{abstract}

\maketitle

\section{Introduction and description of results}\label{sec:introduction}

Homological invariants are very helpful tools for understanding both structure and 
properties of algebraic objects. The most common such invariants used in representation theory of finite
dimensional algebras are projective and injective dimensions that describe the lengths
of the minimal projective resolution and injective coresolution of a module, respectively.
A slightly less common such invariant is the {\em grade} of a module; that is, the 
minimal degree of a non-vanishing extension to a projective module. The latter invariant 
is important in the theory of Auslander regular algebras, see \cite{iyama_marczinzik}.

Incidence algebras of finite posets are important examples of finite dimensional algebras.
The main result of \cite{iyama_marczinzik} asserts that the incidence algebra of a 
finite lattice is Auslander regular if and only if the lattice is distributive. 

Finite Weyl groups play an important role in modern representation theory. They come 
equipped with a natural partial order called the {\em Bruhat order}. Unfortunately,
with the exception of a handful of degenerate cases, the Bruhat order on a Weyl group
is not a lattice. In March 2021, Rene Marczinzik gave a talk at Uppsala Algebra Seminar
in which he addressed the problem  of Auslander regularity of incidence algebras of
Weyl groups with respect to the Bruhat order (to the best of our knowledge, the problem
is still open for symmetric groups in the general case). In connection to this, he 
presented results of computer calculations of grades of simple modules over 
the incidence algebras of the symmetric group in small ranks. 
From these lists, one could observe that the grades of simple modules are often
(but not always) given by Lusztig's $\mathbf{a}$-function from \cite{lusztig}.
In the case of the symemtric group (i.e. in type $A$), 
Lusztig's $\mathbf{a}$-function is uniquely determined by the properties that it is 
constant on all two-sided Kazhdan-Lusztig cells and coincides with the usual length 
function on the longest elements in all parabolic subgroups. In several contexts,
see \cite{KMM,M1,M2}, this function describes homological invariants of algebraic
objects naturally indexed by the  elements of the symmetric group (or, more generally,
of a finite Weyl group). 

The main result of the present paper is the following theorem.

\begin{theorem}\label{thm-main-intro}
The grades of simple modules indexed by boolean permutations, over the incidence algebra 
of the symmetric group with respect to the Bruhat order, are given by Lusztig's 
$\mathbf{a}$-function.
\end{theorem}

We note that the original definition of the $\mathbf{a}$-function reflects some 
subtle numerical properties of the multiplication of the elements in the
Kazhdan-Lusztig bases of the Hecke algebra of a Coxeter group. We do not see 
any immediate connection between the Kazhdan-Lusztig bases of the Hecke algebra
and the incidence algebra  of the symmetric group with respect to the Bruhat order.
Therefore appearance of the $\mathbf{a}$-function in Theorem~\ref{thm-main-intro}
is rather mysterious.

Our proof of this result is combinatorial. Projective resolutions of simple modules
over the incidence algebras of Weyl groups can be constructed using the BGG
complex from \cite{BGG} (that is, the singular homology complex). We use the Serre functor
to relate the grade of the simple module $L_v$, where $v$ is a permutation, 
to the homology of the complex obtained by restricting the BGG complex to the intersection 
$B(v)\cap B(w)$ of two principal ideals in the symmetric group, where $w$ is an arbitrary permutation. 
For boolean $v$, we describe $B(v)\cap B(w)$ in
Proposition~\ref{thm-int-part1} with a more precise version in 
Corollary~\ref{cor:intersecting with boolean elements} under the additional assumption that
$w$ is also boolean. 

Using this explicit description, we proceed with combinatorial analysis of the restricted
BGG complex. In fact, we show that this restricted BGG complex is either exact or 
has exactly one non-zero homology which, moreover, is one-dimensional. For a fixed $v$,
the extreme degree in which such non-zero homology can appear is given by a combinatorial
invariant of $v$ that we call {\em the number of runs in $v$}, introduced in 
Subsection~\ref{sec:matchings.2}. The connection between the homology 
and the number of runs is established in Theorem~\ref{thm:optimal rank and partner for boolean}.

It is an easy combinatorial exercise to show that  the  number of runs for a boolean permutation
coincides with Lusztig's $\mathbf{a}$-function, and this is presented in Section~\ref{sec:tableaux}.
This is, essentially, what one needs to prove Theorem~\ref{thm-main-intro}.

The paper is organized as follows. In Section~\ref{sec:bruhat and boolean},
we collected some basics on the Bruhat order and boolean permutations. 
In Section~\ref{sec:intersections}, we study combinatorics of intersections
$B(v)\cap B(w)$, for boolean $v$, and show how they can be determined either from reduced words or from the permutations' one-line notations. In Section~\ref{sec:motivation}, we describe
in detail the algebraic motivation and setup of the problem we study.
Section~\ref{sec:matchings} contains combinatorial analysis of the homology
of the BGG complex restricted to $B(v)\cap B(w)$. The crucial results of that section give a method for deleting letters from a reduced word without losing the boolean elements in an order ideal, and description of a permutation $w := w(v)$ for which the intersection complex has the desired homology. Finally, in 
Section~\ref{sec:tableaux} we combine all of the pieces necessary for the proof of
Theorem~\ref{thm-main-intro}. In the last section of the paper we briefly address what little 
is known for non-boolean permutations. In particular, we show that Lusztig's $\mathbf{a}$-function 
gives  the grade of the simple module indexed by the longest elements of a parabolic subgroup.
From this we deduce that a simple module is perfect (in the sense that its grade coincides with
its projective dimension) if and only if the index of this module
is the longest element of a parabolic subgroup.

\section{Bruhat order and boolean elements}\label{sec:bruhat and boolean}

The symmetric group $\mf{S}_n$ of permutations of $[1,n]:=\{1,2,\dots,n\}$ 
is a Coxeter group with the natural distinguished set of Coxeter
generators given by the simple reflections $\{\s_i:=(i,i+1) : i \in [1,n-1]\}$. 
A \emph{reduced decomposition} of a permutation $w$ is a product 
$w = \s_{i_1} \cdots \s_{i_{\ell}}$ such that $\ell$ is minimal
(in which case it is called the {\em length} of $w$). 
To save notation, we can equivalently consider \emph{reduced words} of 
a permutation by looking only at the subscripts in a reduced decomposition. 
In this paper, we will let $R(w)$ denote the set of reduced words of 
a permutation $w$. Because both permutations and reduced words can be 
represented by strings of integers, we will write $\rw{s}$ to indicate 
that a string $s$ represents a reduced word. We think of permutations as maps, and we compose maps from right to left.

\begin{example}
For the permutation with one-line notation $4132 \in \mf{S}_4$, we have
$$R(4132) = \{\rw{3213}, \rw{3231}, \rw{2321}\}.$$
\end{example}

Reduced words represent products of simple reflections, so we can use them interchangeably with the permutations they represent. For example, we can write
$$4132 = \rw{3213} = \rw{3231} = \rw{2321}.$$
It is well-known that any two reduced words for a given permutation are related by a sequence of commutation and braid moves \cite{matsumoto, tits}, as we can see in the previous example.

The Bruhat order gives a poset structure to the symmetric group, and it can be defined in terms of reduced words.

\begin{theorem}[{\cite[Theorem 2.2.2]{bjorner brenti}}]\label{thm:subword property}
Let $u,w$ be permutations, and $\rw{s} \in R(w)$. Then $u \le w$ in the Bruhat order if and only if a subword of $\rw{s}$ is a reduced word for $u$.
\end{theorem}

Various structural aspects of this poset have been studied, related to its principal order ideals and intervals (see, for example, \cite{bjorner brenti, bjorner ekedahl, dyer, hultman 1, hultman 2, tenner patt bru, tenner interval factors, tenner bruhat intervals}). Despite this interest and literature, there has been very little attention paid to the intersection of principal order ideals in this poset.

\begin{definition}
For a permutation $w \in \mf{S}_n$, write $B(w)$ for the principal order ideal of $w$ in the Bruhat order. \end{definition}

As studied by Ragnarsson and the second author \cite{ragnarsson tenner 1, ragnarsson tenner 2}, and Hultman and Vorwerk in the case of involutions \cite{hultman vorwerk}, the so-called \emph{boolean} elements of the symmetric group have particularly interesting properties.

\begin{definition}
A permutation $v$ is \emph{boolean} if its principal order ideal $B(v)$ is isomorphic to a boolean algebra.
\end{definition}

Although boolean elements can be defined analogously for any Coxeter group, we are focused on permutations in this work. As shown in \cite{tenner patt bru}, boolean permutations can be characterized in several ways.

\begin{theorem}\label{thm:boolean characterization}
The following statements are equivalent:
\begin{itemize}
\item the permutation $v$ is boolean,
\item the permutation $v$ avoids the patterns $321$ and $3412$, and 
\item reduced words for the permutation $v$ contain no repeated letters.
\end{itemize}
\end{theorem}

In this work we will consider intersections of principal order ideals $B(v) \cap B(w)$, when $v$ is  boolean. First, we will describe the elements of this intersection, and then we will look at its topology.

\section{Intersection ideals}\label{sec:intersections}

\subsection{Orientation and intersecting ideals}\label{sec:intersections.1}

Throughout this section, let $v \in \mf{S}_n$ be a boolean permutation.

\begin{definition}
The \emph{support} of a permutation $w$ is the set of distinct letters appearing in its reduced words. This will be denoted $\supp(w)$.
\end{definition}

Thus a permutation is boolean if and only if its length is equal to the size of its support. For an arbitrary permutation $w$, we can make the following observation about how $\supp(w)$ might impact the poset $B(w)$.

Let $x$ be a word and $I$ a subset of letters appearing in $x$. 
We will denote by $x_I$ the subword of $x$ consisting of all letters from $I$. 

\begin{lemma}\label{lem:commuting generators make product poi}
Suppose that $w$ is a permutation with $\supp(w) = X \sqcup Y$, such that either
\begin{itemize}
\item $x$ and $y$ commute for all $x \in X$ and $y \in Y$, in which case let $\rw{s} \in R(w)$ be any reduced word; or
\item there is exactly one pair of noncommuting letters $(x_0,y_0) \in X \times Y$, and there exists $\rw{s} \in R(w)$ in which all appearances of $x_0$ are to the left of all appearances of $y_0$.
\end{itemize}
Then both $\rw{s_X}$ and $\rw{s_Y}$ are reduced words, and $B(w)$ can be written as a direct product:
$$B(w) \cong B(\rw{s_X}) \times B(\rw{s_Y}).$$
\end{lemma}

\begin{proof}
The $\rw{s} \in B(w)$ described in the statement of the lemma can be transformed via commutations into $\rw{s_Xs_Y} \in B(w)$. The result follows.
\end{proof}

Theorem~\ref{thm:boolean characterization} means that Lemma~\ref{lem:commuting generators make product poi} can be applied to any boolean permutation.

The goal of this section is to describe $B(v) \cap B(w)$ for arbitrary permutations $w$. Because $v$ is boolean, the ideal $B(v)$ is determined by:
\begin{itemize}
\item the support of $v$,
\item the pairs of noncommuting (i.e., consecutive) letters appearing in the support, and 
\item the order in which noncommuting letters appear in elements of $R(v)$.
\end{itemize}
Note that this last item is well-defined because all letters are distinct. Therefore, if $k$ appears to the left of $k+1$ in one element of $R(v)$, then in fact $k$ appears to the left of $k+1$ in all elements of $R(v)$.

When looking at the principal order ideal of a boolean permutation $v$, we might want to consider the permutations covering an element $u$ in that ideal. By Theorems~\ref{thm:subword property} and~\ref{thm:boolean characterization} we can think of such an element as being obtained from some $\rw{s} \in R(u)$ by inserting a letter $k$ in such a way as to be consistent with elements of $R(v)$. The lack of repeated letters among those elements means that there is no ambiguity about how to insert $k$. When such an element exists, we will write it as
$$u \addletter \s_k.$$

The support of a permutation can be read off from any of its reduced words. It can also be detected from the one-line notation of the permutation, as described in the following lemma.

\begin{lemma}[{cf.~\cite[Lemma 2.8]{tenner rep patt}}]\label{lem:support detection}
For any $w \in \mf{S}_n$, the following statements are equivalent:
\begin{itemize}
\item $k \in \supp(w)$,
\item $\{w(1),\ldots,w(k)\} \neq \{1,\ldots, k\}$, and
\item $\{w(k+1),\ldots, w(n)\} \neq \{k+1,\ldots, n\}$.
\end{itemize}
\end{lemma}

Several facts about $B(v) \cap B(w)$ follow directly from Theorem~\ref{thm:subword property}.

\begin{lemma}\label{lem:intersection atoms and maxes}
Consider $v,w \in \mf{S}_n$, where $v$ is boolean.
\begin{enumerate}\renewcommand{\labelenumi}{(\alph{enumi})}
\item The length $1$ elements of $B(v) \cap B(w)$ are $\supp(v) \cap \supp(w)$.
\item Each element of $B(v) \cap B(w)$ is itself a boolean permutation.
\item The intersection $B(v) \cap B(w)$ is an order ideal.
\end{enumerate}
\end{lemma}

Part (c) of Lemma~\ref{lem:intersection atoms and maxes} means that we will understand $B(v) \cap B(w)$ once we can describe its maximal elements. To do this, we must understand which subsets of $\supp(v) \cap \supp(w)$ will describe an element of $B(v) \cap B(w)$. The only concern arises from consecutive letters in $\supp(v) \cap \supp(w)$. If there are such letters, then they appear in a particular order in all elements of $R(v)$. If they can appear in the same order in an element of $R(w)$, then they can appear together in an element of $B(v) \cap B(w)$. If they never appear in that same order, in any elements of $R(w)$, then these two letters cannot appear together in any element of $B(v) \cap B(w)$.

\begin{definition}
Consider a permutation $w$ with $k,k+1 \in \supp(w)$. If all appearances of $k$ are to the left of all appearances of $k+1$ in reduced words for $w$, then $k$ and $k+1$ have \emph{increasing orientation} in $w$. If all appearances of $k$ are to the right of all appearances of $k+1$, then they have \emph{decreasing orientation}. Otherwise, their orientation is \emph{interlaced}. Two orientations \emph{match} unless one is increasing and the other is decreasing.
\end{definition}

\subsection{Maximal selfish subsets}\label{sec:intersections.2}

For a positive integer $k$, let us consider the set of integers $[1,k]$. Denote by
$\mathscr{Q}_k$ the set of all subsets $X\subset [1,k]$ that are maximal with respect to inclusions and that have the following property, which we call {\em selfishness}:
\begin{itemize}
\item if $i\in X$, then $i\pm 1\not\in X$.
\end{itemize}
Here is the list of $\mathscr{Q}_k$, for $k=1,2,3,4,5$:
\begin{displaymath}
\begin{array}{rcl}
\mathscr{Q}_1&=&\big\{\{1\}\big\},\\ 
\mathscr{Q}_2&=&\big\{\{1\},\{2\}\big\},\\ 
\mathscr{Q}_3&=&\big\{\{1,3\},\{2\}\big\},\\ 
\mathscr{Q}_4&=&\big\{\{1,3\},\{2,4\},\{1,4\}\big\}, \text{ and}\\ 
\mathscr{Q}_5&=&\big\{\{1,3,5\},\{2,5\},\{2,4\},\{1,4\}\big\}.
\end{array}
\end{displaymath}
We denote by $\mathscr{Q}'_k$ the set of all $X\in \mathscr{Q}_k$ such that $k\in X$
and we set $\mathscr{Q}''_k:=\mathscr{Q}_k\setminus\mathscr{Q}'_k$.

\begin{proposition}\label{prop-wm11}
{\hspace{1mm}}

\begin{enumerate}[label=(\alph{enumi})]
\item\label{prop-wm11.1} For $k>2$, the map $X\mapsto X\setminus\{k\}$ is a bijection from
$\mathscr{Q}'_k$ to $\mathscr{Q}_{k-2}$.
\item\label{prop-wm11.2} For $k>3$, the map $X\mapsto X\setminus\{k-1\}$ is a bijection from
$\mathscr{Q}''_k$ to $\mathscr{Q}_{k-3}$.
\item\label{prop-wm11.3} For $k>3$, we have $|\mathscr{Q}_k|=|\mathscr{Q}_{k-2}|+|\mathscr{Q}_{k-3}|$.
\end{enumerate}
\end{proposition}

\begin{proof}
If $k\in X$, then $k-1\not\in X$ due to selfishness. Therefore $Y:=X\setminus\{k\}$ is a selfish
subset of $[1,k-2]$. If $Y$ were not maximal, we would be able to 
add to it some element $i\in [1,k-2]$ preserving its selfishness. Since $k-1\not\in X$, the subset $X\cup\{i\}$ would also be selfish,
contradicting the maximality of $X$. Therefore $Y\in \mathscr{Q}_{k-2}$.

Conversely, given $Y\in \mathscr{Q}_{k-2}$, the set $X:=Y\cup\{k\}$ is, clearly, selfish. By
the maximality of $Y$, we cannot add to $X$ any $i\in  [1,k-2]$ without violating
selfishness. Clearly, we cannot add $k-1$ either since $k\in X$. Therefore
$X\in \mathscr{Q}'_k$. This proves Claim~\ref{prop-wm11.1}.

Note that $X\in\mathscr{Q}''_k$ implies $k-1\in X$,
for otherwise $X\cup\{k\}$ would be selfish, contradicting the maximality of $X$.
Therefore $\mathscr{Q}''_k=\mathscr{Q}'_{k-1}$ and
Claim~\ref{prop-wm11.2} follows from Claim~\ref{prop-wm11.1}
applied to $\mathscr{Q}'_{k-1}$. 
Claim~\ref{prop-wm11.3} follows
from Claims~\ref{prop-wm11.1} and \ref{prop-wm11.2}.
\end{proof}

From Proposition~\ref{prop-wm11}\ref{prop-wm11.3}, it follows that the
sequence $\{|\mathscr{Q}_{k}|\,:\,k\geq 1\}$ is the (suitably offsetted) 
Padovan sequence \cite[A000931]{oeis}. 

The above concept has a natural generalization to an arbitrary
finite subset $\mathcal{U}$ of $\mathbb{Z}_{>0}$ (a {\em universe}).
We can consider selfish subsets of such a $\mathcal{U}$,
maximal with respect to inclusions. We denote the set of all such
subsets by $\mathscr{Q}(\mathcal{U})$. The set $\mathcal{U}$ has 
a unique decomposition 
\begin{displaymath}
\mathcal{U}=\coprod_i  \mathcal{U}_i
\end{displaymath}
of minimal length into a disjoint union of interval subsets
(i.e., each $\mathcal{U}_i$ is of the form $[p,p+q]$).
The minimality of the length means that $\mathcal{U}_i\cup\mathcal{U}_j$
is not an interval subset for any $i\neq j$.

It is clear that maximal selfish subsets of $\mathcal{U}$ are just 
the unions of maximal selfish subsets of the individual $\mathcal{U}_i$'s,
and the maximal selfish subsets in each $\mathcal{U}_i$ are completely
described by Proposition~\ref{prop-wm11}.

\subsection{Intersections with principal ideals of boolean elements}\label{sec:intersections.3}

Our next step is to describe $B(v) \cap B(w)$, for an arbitrary permutation $w$. 
As initially presented, this will require understanding properties about the 
reduced words for $v$ and $w$. However, following the theorem, we will show 
how it can also be determined from the one-line notations for $v$ and $w$.

Let $v,w \in \mf{S}_n$ be such that $v$ is boolean and fix some $\rw{s} \in R(v)$. Denote by $\mathcal{W}(v,w)$ the set of all minimal 
subwords of $\rw{s}$ of the form $\rw{i(i+1)\cdots(i+j)}$
or $\rw{(i+j)\cdots(i+1)i}$ with the following properties:
\begin{itemize}
\item if $\rw{i(i+1)\cdots(i+j)}$ is a subword of $\rw{s}$,
then $\rw{i(i+1)\cdots(i+j)}\not\leq w$ while both
$\rw{i(i+1)\cdots(i+j-1)}\leq w$ and $\rw{(i+1)(i+2)\cdots(i+j)}\leq w$;
\item if $\rw{(i+j)\cdots(i+1)i}$ is a subword of $\rw{s}$,
then $\rw{(i+j)\cdots(i+1)i}\not\leq w$ while both
$\rw{(i+j-1)\cdots(i+1)i}\leq w$ and $\rw{(i+j)\cdots(i+2)(i+1)}\leq w$.
\end{itemize}
Further, denote by $\mathcal{W}(v,w)^\uparrow$ the set of all maximal subwords
of $\rw{s}$ which do not have any of the elements in $\mathcal{W}(v,w)$
as subwords. In other words, $\mathcal{W}(v,w)^\uparrow$ is the set of maximal
elements in the complement to the filter generated by $\mathcal{W}(v,w)$
inside $B(v)$. Note that $\mathcal{W}(v,w)$ contains all simple reflections from
$\supp(v) \setminus \supp(w)$.

For example, if $v=\rw{321}$ and $w=\rw{2132}$, then we  have 
$\mathcal{W}(v,w)=\{\rw{321}\}$ and, consequently, 
$\mathcal{W}(v,w)^\uparrow=\{\rw{21},\rw{31},\rw{32}\}$. Similarly, if $v = \rw{32145}$ and $w = \rw{4521324}$, then $\mathcal{W}(v,w) = \{\rw{321},\rw{345}\}$ and $\mathcal{W}(v,w)^\uparrow = \{\rw{2145},\rw{314},\rw{315},\rw{324},\rw{325}\}$.

\begin{proposition}\label{thm-int-part1}
For $v,w \in \mf{S}_n$ with $v$ boolean and $\rw{s} \in R(v)$, 
the set $\mathcal{W}(v,w)^\uparrow$ is exactly the set of all
maximal elements in $B(v) \cap B(w)$.
\end{proposition}

\begin{proof}
Each minimal element in $\mf{S}_n\setminus B(w)$ is, by definition,  join irreducible.
It is well-known, see \cite{LS}, that join irreducible elements in $\mf{S}_n$
are exactly the elements with unique left descent and unique right descent
(i.e. the so-called bigrassmannian elements, see also \cite{KMM2,KMM3} for more details). 
From the classification of all bigrassmannian elements in $\mf{S}_n$ (see, for example,
\cite[Figure~7]{KMM2}), it follows that the boolean bigrassmannian elements are exactly 
the elements of the form $\rw{i(i+1)\cdots(i+j)}$ or $\rw{(i+j)\cdots(i+1)i}$.

This implies that the minimal  elements in the complement  to 
$B(v) \cap B(w)$ in $B(v)$ are all of the form $\rw{i(i+1)\cdots(i+j)}$ or 
$\rw{(i+j)\cdots(i+1)i}$, for some $i$ and $j$. Now the claim of the 
proposition follows directly from the definitions of the sets
$\mathcal{W}(v,w)$ and $\mathcal{W}(v,w)^\uparrow$.
\end{proof}

In the special case of the intersection of two boolean principal ideals, 
we can make the statement of Theorem~\ref{thm-int-part1} more precise.
Fix permutations $v,w \in \mf{S}_n$ such that $v$ is boolean. 
Choose any element $\rw{s} \in R(v)$. Denote by 
$\mathcal{V}(v,w)$ the set of
all letters $k$ for which there is $x\in\{k\pm 1\}$ such that
the orientation of the pair $\{k,x\}$ in $v$ and $w$ does not match.

\begin{corollary}\label{cor:intersecting with boolean elements}
For $v,w \in \mf{S}_n$ with both $v$ and $w$ boolean and $\rw{s} \in R(v)$, 
the maximal elements in the order ideal 
$B(v) \cap B(w)$ are exactly the subwords of $\rw{s}$ whose support is of the form
\begin{equation}\label{eq-wm-02}
\big((\supp(v) \cap \supp(w))\setminus \mathcal{V}(v,w)\big)\cup X,
\quad\text{ for some }\quad X\in\mathscr{Q}(\mathcal{V}(v,w)).
\end{equation}
\end{corollary}

\begin{proof}
For boolean $w$, the two conditions
\begin{displaymath}
\rw{i(i+1)\cdots(i+j-1)}\leq w\text{ and }\rw{(i+1)(i+2)\cdots(i+j)}\leq w, \text{ with }j>1, 
\end{displaymath}
imply $\rw{i(i+1)\cdots(i+j)}\leq w$. Similarly, the two conditions
\begin{displaymath}
\rw{(i+j-1)\cdots(i+1)i}\leq w\text{ and }\rw{(i+j)\cdots(i+2)(i+1)}\leq w, \text{ with }j>1, 
\end{displaymath}
imply $\rw{(i+j)\cdots(i+1)i}\leq w$.  This means that the set
$\mathcal{W}(v,w)$ consists of elements of the form
$[i(i+1)]$ or $[(i+1)i]$. It follows directly from the definitions 
that Formula~\eqref{eq-wm-02} describes exactly the elements of $\mathcal{W}(v,w)^\uparrow$.
Now the claim follows from Proposition~\ref{thm-int-part1}.
\end{proof}

In fact, the previous argument shows that we can use this construction whenever there are no minimal subwords having $j>1$, using the notation at the beginning of this section.

\begin{corollary}\label{cor:intersecting with elements that always have j=1}
For $v,w \in \mf{S}_n$ with $v$ boolean and $\rw{s} \in R(v)$, if all elements of $\mathcal{W}(v,w)$ have $j=1$, then the maximal elements in the order ideal 
$B(v) \cap B(w)$ are exactly the subwords of $\rw{s}$ whose support is of the form
\begin{equation}\label{eq-max element construction when j=1}
\big((\supp(v) \cap \supp(w))\setminus \mathcal{V}(v,w)\big)\cup X,
\quad\text{ for some }\quad X\in\mathscr{Q}(\mathcal{V}(v,w)).
\end{equation}
\end{corollary}

\subsection{The same property via one-line notation}\label{sec:intersections.4}

The relative order(s) of $k$ and $k+1$ can be read off by looking at reduced words of a permutation, but it would be nice if they could also be determined directly from its one-line notation. We use the next results to show how that can be done.

\begin{proposition}\label{prop:letters interlace or not}
Consider a permutation $w$, with $\{k,k+1\} \subseteq \supp(w)$. Then, exactly one of the following possibilities holds.
\begin{enumerate}\renewcommand{\labelenumi}{(\roman{enumi})}
\item $k$ and $k+1$ are interlaced in $w$, meaning that $w$ has a reduced word with one of the following forms:
$$\rw{\cdots k \cdots (k+1) \cdots k \cdots} \hspace{.25in} \text{or} \hspace{.25in} \rw{\cdots (k+1) \cdots k \cdots (k+1) \cdots}$$
\item $k$ and $k+1$ have either increasing or decreasing orientation in $w$, meaning that $w$ has a reduced word with one of the following forms:
$$\rw{\Big( \ \text{letters } \le k \ \Big)\Big( \ \text{letters } \ge k+1 \ \Big)} \hspace{.25in} \text{or} \hspace{.25in} \rw{\Big( \ \text{letters } \ge k+1 \ \Big)\Big( \ \text{letters } \le k \ \Big)}$$
\end{enumerate}
\end{proposition}

\begin{proof}
Reduced words for $w$ contain both $k$ and $k+1$. Suppose that some $\rw{s} \in R(w)$ interlaces $k$ and $k+1$. Then no sequence of commutation and braid moves can produce a word in which all appearances of $k$ are to one side of all appearances of $k+1$, so $w$ has no reduced words of the forms shown in (ii). Alternatively, suppose that no elements of $R(w)$ interlace $k$ and $k+1$. Thus, in each element of $R(w)$, all appearances of $k$ are to one side of all appearances of $k+1$. Suppose, without loss of generality, that $\rw{s} \in R(w)$ has the form $s = \alpha \beta$, where $\alpha$ contains $k$ but not $k+1$, and $\beta$ contains $k+1$ but not $k$. Then, as in \cite{tenner rwm}, we can use commutations to push all $x > k+1$ in $\alpha$ to the right and all $x < k$ in $\beta$ to the left, in order to find an element of $R(w)$ having one of the forms depicted in (ii).
\end{proof}

As discussed previously, these increasing or decreasing orientations restrict which elements can appear in an intersection of principal order ideals. We now describe how to identify a permutation having this property.

\begin{theorem}\label{thm:detecting no interlace}
Consider a permutation $w \in \mf{S}_n$, with $\{k,k+1\} \subseteq \supp(w)$.
\begin{enumerate}\renewcommand{\labelenumi}{(\alph{enumi})}
\item 
A reduced word for $w$ has the form
$$\rw{\Big( \ \text{letters } \le k \ \Big)\Big( \ \text{letters } \ge k+1 \ \Big)}$$
(that is, $k$ and $k+1$ have increasing orientation in $w$) if and only if there exists $x < k+1$ such that
 $\{w(1),\cdots,w(k)\} = [1,k+1] \setminus \{x\}$ and $w^{-1}(x) > k+1$.
\item 
A reduced word for $w$ has the form
$$\rw{\Big( \ \text{letters } \ge k+1 \ \Big)\Big( \ \text{letters } \le k \ \Big)}$$
(that is, $k$ and $k+1$ have decreasing orientation in $w$) if and only if there exists $y > k+1$ such that
$\{w(k+2),\ldots, w(n)\} = [k+1,n] \setminus \{y\}$ and $w^{-1}(y) < k+1$.
\item 
The letters $k$ and $k+1$ are interlaced in $w$ if and only if there exists $i \in [1,k]$ and $j \in [k+2,n]$ such that $w(i) > k+1$ and $w(j) < k+1$.
\end{enumerate}
\end{theorem}

\begin{proof}
Consider part (a) of the theorem.

Suppose that $w$ has such a reduced word. This reduced word indicates that $w = tu$, where $\supp(t) \subseteq [1,k]$ and $\supp(u) \subseteq [k+1,n-1]$. Moreover, because $\{k,k+1\} \subseteq \supp(w)$, we must have $k \in \supp(t)$ and $k+1 \in \supp(u)$. The permutation $u$ fixes all $i < k+1$. On the other hand, by Lemma~\ref{lem:support detection}, $u(k+1) > k+1$ and $u^{-1}(k+1) = y > k+1$. Similarly, the permutation $t$ fixes all $i > k+1$, with $t(k+1) < k+1$ and $t^{-1}(k+1) = x < k+1$. Thus in the product $w = tu$, we have
$$w(y) = tu(y) = t(k+1) < k+1.$$
Moreover, for all $z \neq y$, either $u$ fixes $z$ or $t$ fixes $u(z)$. In particular, $w(x) = tu(x) = t(x) = k+1$, completing the proof of this direction.

Now suppose that there exists $x < k+1$ such that
 $\{w(1),\cdots,w(k)\} = [1,k+1] \setminus \{x\}$ and $i:= w^{-1}(x) > k+1$. Consider the permutation
\begin{equation}\label{eqn:detecting no interlace, in proof}
u := (\s_k \s_{k-1} \cdots \s_x) w (\s_{i-1} \s_{i-2} \cdots \s_{k+1}).
\end{equation}
The factor on the right in Equation~\eqref{eqn:detecting no interlace, in proof} slides $x$ leftward in the one-line notation for $w$, swapping it with $w(j) > k+1 > x$ at each step, until $x$ is sitting in position $k+1$ of the permutation. Recall the hypotheses on $w$. The factor on the left in Equation~\eqref{eqn:detecting no interlace, in proof} swaps the value $x$ with the value $x+1$ in the one-line notation, then $x+1$ with $x+2$, and so on, always moving the larger value into position $k+1$ from somewhere to the left of that position. Therefore $\ell(u) = \ell(w) - (k-x+1) - (i-k-1) = \ell(w) + x - i$. After all of these transpositions, the resulting permutation $u$ fixes $k+1$, and $u(i) \in [1,k]$ for all $i \in [1,k]$. Therefore, by Lemma~\ref{lem:support detection}, we have $k,k+1 \not\in \supp(u)$, and thus there exists some $\rw{\alpha\beta} \in R(u)$ in which $\alpha$ contains only letters less than $k$ and $\beta$ contains only letters that are greater than $k+1$. Hence
$$\rw{x(x+1)\cdots (k-1)k \alpha \beta (k+1)(k+2) \cdots (i-2)(i-1)} \in R(w),$$
and this is the desired reduced word.

Part (b) follows from part (a) using conjugation by $w_0$ (which acts on $S_n$
by reversing the one-line notation).

Now consider part (c) of the theorem.

Suppose, first, that there is no such $i$. Because $\{k,k+1\} \subseteq \supp(v)$, Lemma~\ref{lem:support detection} says that $\{w(1), \ldots, w(k)\} \neq [1,k]$ and $\{w(1),\ldots, w(k+1)\} \neq [1,k+1]$. Thus, if there is no such $i$, then we must have $w(h) = k+1$ for some $h < k+1$, $w(k+1) > k+1$, and $w(q) \in [1,k]$ for all $q \in [1,k] \setminus \{h\}$. But then $w$ has the form described in part (a) of the current theorem, and so, by Proposition~\ref{prop:letters interlace or not}, the letters $k$ and $k+1$ are not interlaced in $w$. Similarly, if there is no such $j$, then the result will follow from part (b) of the current theorem.

Now suppose that there are such $i$ and $j$. Then the permutation $w$ has neither form from parts (a) or (b) of the current theorem, so, by Proposition~\ref{prop:letters interlace or not}, the letters $k$ and $k+1$ are interlaced in $w$.
\end{proof}

Lemma~\ref{lem:support detection} gave a method for detecting the support of a permutation from its one-line notation. Theorem~\ref{thm:detecting no interlace} gives a method for determining the orientation of any $\{k,k+1\} \subseteq \supp(w)$ from the one-line notation of $w$, as well. In some ways, this is an analogy to Theorem~\ref{thm:boolean characterization}, which equates reduced word properties with pattern-avoiding (one-line notation) properties. 

We can use this orientation detection to construct the maximal elements of $B(v) \cap B(w)$ when $v$ and $w$ are both boolean, following Corollary~\ref{cor:intersecting with boolean elements}. Moreover, by Corollary~\ref{cor:intersecting with elements that always have j=1}, we can also use it with no conditions on $w$, when all subwords discussed at the beginning of Section~\ref{sec:intersections.3} have $j=1$.

\begin{example}\label{ex:312647895}
Consider the boolean permutation $v = 312647895 \in \mf{S}_9$ 
and the non-boolean permutation $w = 325184769 \in \mf{S}_9$, both written 
in one-line notation. Using Lemma~\ref{lem:support detection}, we can compute
$$\supp(v) = \{1,2,4,5,6,7,8\} \hspace{.25in} \text{and} \hspace{.25in} \supp(w) = \{1,2,3,4,5,6,7\}.$$
The intersection of these sets is $\{1,2,4,5,6,7\}$, so we check four orientations 
using Theorem~\ref{thm:detecting no interlace}.
\begin{center}
{\renewcommand{\arraystretch}{1.5}\begin{tabular}{c||c|c}
\vspace{-.1in}Consecutive & Orientation & Orientation\\
generators & in $v$ & in $w$\\
\hline
\hline
$\{\s_1,\s_2\}$ & decreasing & interlaced\\
\hline
$\{\s_4,\s_5\}$ & decreasing & increasing\\
\hline
$\{\s_5,\s_6\}$ & increasing & decreasing\\
\hline
$\{\s_6,\s_7\}$ & increasing & interlaced
\end{tabular}}
\end{center}
The middle column of the table tells us that the only subword we need to worry about is $\rw{567}$, but the rightmost column shows that, in fact, the subwords discussed at the beginning of Section~\ref{sec:intersections.3} all have $j=1$. Thus we can use Corollary~\ref{cor:intersecting with elements that always have j=1}, with $\mathcal{V}(v,w) = \{4,5,6\}$. Therefore, the maximal elements of $B(v) \cap B(w)$ are defined from any 
$\rw{s} \in R(v)$ by deleting $8$ (which is not in $\supp(w)$), 
and then by deleting the complement of a maximal selfish subset of 
$\{4,5,6\}$ (i.e., either deleting $5$ or deleting both $4$ and $6$).
So if we take $\rw{s} = \rw{5214678}$, then the maximal elements of the intersection are defined by
$$\rw{521\xcancel{4} \xcancel{6}7 \xcancel{8}} = 312465879 \hspace{.25in} \text{and} \hspace{.25in} \rw{\xcancel{5}21467 \xcancel{8}} = 312547869.$$
In other words,
$$B(312647895) \cap B(325184769) = B(\rw{5217}) \cup B(\rw{21467} ).$$
\end{example}

The most extreme case of non-matching orientation is that of the boolean element
$v$ and $w=v^{-1}$, here is an example.

\begin{example}
Consider $v = 312647895 = \rw{5214678} \in \mf{S}_9$, as in Example~\ref{ex:312647895}. 
Then, following Subsection~\ref{sec:intersections.2} and Corollary~\ref{cor:intersecting with boolean elements}, 
we have $\supp(v) = [1,2]\cup [4,8]$. The eight maximal elements of $B(v) \cap B(v^{-1})$ have reduced decompositions defined by the product $\{1,2\} \times \{468, 47, 57, 58\}$, as discussed in Section~\ref{sec:intersections.2}. That is,
\begin{align*}
B(v) \cap B(v^{-1}) = B(\rw{1468}) &\cup B(\rw{147}) \cup B(\rw{157}) \cup B(\rw{158})\\
 \cup& B(\rw{2468}) \cup B(\rw{247}) \cup B(\rw{257}) \cup B(\rw{258}).
\end{align*}
\end{example}

\section{Motivation:  Incidence algebras and grades of simple modules}\label{sec:motivation}

\subsection{Incidence algebras and their modules}\label{sec:motivation.1}

Let us fix an algebraically closed field $\Bbbk$. As usual, we denote by $*$
the classical $\Bbbk$-duality $\mathrm{Hom}_\Bbbk({}_-,\Bbbk)$.

Let $(\mathbf{P},\prec)$ be a finite poset. 
Consider the incidence algebra $\mathcal{I}(\mathbf{P})$ over $\Bbbk$. The 
algebra $\mathcal{I}(\mathbf{P})$ can be described by its  Gabriel quiver $\Gamma$ that has
\begin{itemize}
\item the elements of $\mathbf{P}$ as vertices;
\item the arrows $p\to q$, for each pair $(p,q)\in \mathbf{P}^2$ such that $p$ covers $q$;
\end{itemize}
and the relations that, for any $(p,q)\in \mathbf{P}^2$, all paths from $p$ to $q$ coincide.

As usual, the simple $\mathcal{I}(\mathbf{P})$-modules are in bijection with 
the elements in $\mathbf{P}$. Given $p\in \mathbf{P}$, the corresponding simple module $L_p$ 
is one-dimensional at $p$ and zero-dimensional at all other vertices. Furthermore,
all arrows from the Gabriel quiver act on $L_p$ as the zero linear maps.

The indecomposable projective cover $P_p$ of $L_p$ is supported on the 
ideal $\mathbf{P}_{\preceq p}$, is one-dimensional at each point of this ideal and
zero-dimensional at all other points, and all arrows between the elements of $\mathbf{P}_{\preceq p}$
act as the identity linear transformations.

Dually, the indecomposable injective envelope $I_p$ of $L_p$ is supported on the 
coideal (that is, a filter) $\mathbf{P}_{\succeq p}$, is one-dimensional at each point of this coideal and
zero-dimensional at all other points, and all arrows between the elements of $\mathbf{P}_{\succeq p}$
operate as the identity linear transformations. See Subsection~\ref{sec:motivation.5} for an example.

We denote by $\mathcal{I}(\mathbf{P})$-mod the category of finite dimensional 
(left) $\mathcal{I}(\mathbf{P})$-modules, which we identify with the category of 
modules over the above quiver satisfying the above relations.

\subsection{Grades of simple modules}\label{sec:motivation.2}

Since the Gabriel quiver of $\mathcal{I}(\mathbf{P})$ is acyclic, the algebra $\mathcal{I}(\mathbf{P})$
has finite global dimension. In particular, for each $0\neq M\in \mathcal{I}(\mathbf{P})$-mod, 
the following invariant, called  {\em grade}, is well-defined and finite:
\begin{displaymath}
\mathbf{grade}(M):=\min\{i\,:\, 
\mathrm{Ext}_{\mathcal{I}(\mathbf{P})}^i(M,\mathcal{I}(\mathbf{P}))\neq 0\}\leq 
\mathrm{proj.dim}(M).
\end{displaymath}

Of special interest for us will be the grades of the simple modules $L_p$, where $p\in\mathbf{P}$.

\subsection{Grades via the Serre functor}\label{sec:motivation.3}

Consider the bounded derived category $\mathcal{D}^b(\mathcal{I}(\mathbf{P}))$ of $\mathcal{I}(\mathbf{P})$-mod.
Since $\mathcal{I}(\mathbf{P})$ has finite global dimension, the category $\mathcal{D}^b(\mathcal{I}(\mathbf{P}))$
has a Serre functor $\mathbb{S}$ given by the left derived functor of tensoring with the dual bimodule
$\mathcal{I}(\mathbf{P})^*$. The functor $\mathbb{S}$ is a self-equivalence  of $\mathcal{D}^b(\mathcal{I}(\mathbf{P}))$, see \cite{bondal_kapranov,happel}.
We have $\mathbb{S}P_p\cong I_p$, for each $p\in\mathbf{P}$.

\begin{lemma}\label{lem.wm-1}
For $0\neq M\in \mathcal{I}(\mathbf{P})$, the grade of $M$ coincides with 
the minimal $i$ such that 
the $-i$-th homology  of the complex $\mathbb{S}M$ is non-zero.
\end{lemma}

\begin{proof}
By definition, the grade of  $M$ is the minimal value of $i$ such that
\begin{displaymath}
\mathrm{Hom}_{\mathcal{D}^b(\mathcal{I}(\mathbf{P}))}(M,\mathcal{I}(\mathbf{P})[i])\neq 0. 
\end{displaymath}
Applying the equivalence $\mathbb{S}$, we obtain
\begin{displaymath}
\mathrm{Hom}_{\mathcal{D}^b(\mathcal{I}(\mathbf{P}))}(\mathbb{S}M,\mathbb{S}\mathcal{I}(\mathbf{P})[i])\neq 0. 
\end{displaymath}
It remains to note that $\mathbb{S}\mathcal{I}(\mathbf{P})$ is an injective cogenerator of
$\mathcal{I}(\mathbf{P})$-mod and hence taking homomorphisms into it detects the homology.
\end{proof}

Lemma~\ref{lem.wm-1} suggests that, to determine the grade of $L_p$, one needs to take a minimal
projective resolution of $L_p$, apply $\mathbb{S}$ to it and then understand the homology of
the obtained complex.

\subsection{Incidence algebras for Bruhat posets of finite Weyl groups}\label{sec:motivation.4}
Assume now that $\mathrm{char}(\Bbbk)=0$.

Let $W$ be a {\em finite Weyl group} and $S$ a fixed set of  simple  reflections in $W$.
Then $W$ is a poset with respect to the {\em Bruhat order} $\leq$. This poset has the
minimum element $e$ and the maximum element $w_0$, the longest element of $W$. 
We denote by $\ell:W\to \mathbb{Z}_{\geq 0}$
the associated {\em length function}. For simplicity, we denote by $A$ the
$\Bbbk$-algebra $\mathcal{I}((W,\leq))$. The algebra $A$ is the main protagonist in our motivation.

We would like to determine $\mathbf{grade}(L_w)$, for each $w\in W$. 
Taking into account the observations in the previous subsection, 
let us start with a description of projective resolutions of 
the modules $L_w$, where $w\in W$. 

For $i\geq 0$, denote by $V_i$ the formal vector space with basis 
$\{v_w\,:\,\ell(w)=i\}$. By \cite{BGG}, for each $i$, there exists 
a linear map $d_i:V_i\to V_{i+1}$ such that
\begin{itemize}
\item the $v_x$-$v_y$-coefficients of $d_i$ is non-zero if and only if
$y\leq x$;
\item all such non-zero coefficients are $\pm 1$;
\item $d_{i+1}\circ d_{i}=0$, for all $i$.
\end{itemize}
The associated complex 
\begin{displaymath}
0\to V_0\to V_1\to \dots\to V_{\ell(w_0)}\to 0 
\end{displaymath}
is exact and is called a {\em BGG complex}. It has the property that 
its restriction to the part supported at a principal (co)ideal is 
exact (unless it is the ideal of the minimum element, respectively 
the coideal of the maximum element). This complex can also be interpreted
as the singular homology complex for $W$.

For $i\geq 0$, denote by $Q(w,i)$ the direct sum of all $P_x$, where $x\leq w$
and $\ell(w)-\ell(x)=i$. 

\begin{proposition}\label{propwm-2}
There is a projective resolution of $L_w$ of the form
\begin{equation}\label{eq-wm-3}
\dots\to Q(w,2)\to Q(w,1)\to Q(w,0)\to 0,
\end{equation}
where, for a summand $P_x$ in $Q(w,i)$ and $P_y$ in $Q(w,i-1)$ such that $x\leq y$,
the map from $P_x$ to $P_y$ is given by the corresponding coefficient in the BGG complex.
\end{proposition}

\begin{proof}
This follows directly from the properties of the BGG complex listed above.
\end{proof}

For $i\geq 0$, denote by $F(w,i)$ the direct sum of all $I_x$, where $x\leq w$
and $\ell(w)-\ell(x)=i$. Applying $\mathbb{S}$ to \eqref{eq-wm-3} results in
a complex 
\begin{equation}\label{eq-wm-4}
\dots\to F(w,2)\to F(w,1)\to F(w,0)\to 0,
\end{equation}
with the property that, for a summand $I_x$ in $Q(w,i)$ and $I_y$ in $Q(w,i-1)$ 
such that $x\leq y$, the map from $I_x$ to $I_y$ is given by the corresponding 
coefficient in the BGG complex.

Now we want to understand the homology of \eqref{eq-wm-4}, or, more precisely, the 
rightmost degree in which  non-zero homology appears. This determines the
grade of $L_w$.

For $u\in W$, let us restrict \eqref{eq-wm-4} to the vertex $u$. Recall that 
the injective module $I_x$, for $x\in W$, is supported at $W_{\geq x}$.
Therefore, each $I_x\in Q(w,i)$ such that $x\leq u$ contributes one dimension
for the vertex $u$ at position $-i$ of the complex \eqref{eq-wm-4}.
That is, the restriction of \eqref{eq-wm-4} to $u$ is the complex
\begin{equation}\label{eq-wm-5}
\dots\to F(w,2)_u\to F(w,1)_u\to F(w,0)_u\to 0,
\end{equation}
where $F(w,2)_u$ is the sum of one-dimensional spaces indexed by $x$ such that
$x\leq w$, $x\leq u$ and $\ell(w)-\ell(x)=i$, and the differential is the
restriction of the differential from the BGG complex. In other words, this is
exactly the restriction of the BGG complex to $W_{\leq w}\cap W_{\leq u}$, and our goal is to minimize,
over all possible $u$, the absolute value of the rightmost degree in which this 
complex has non-zero homology.

Our setup is such that the vertex $w$ is placed at the homological position $0$.
In particular, the vertex $e$ is placed at position $-\ell(w)$. Taking
$u=e$, we get a complex concentrated in position $-\ell(w)$, which means that
$\ell(w)$ is an upper bound for our answer.

If $e\neq w\leq u$, then $W_{\leq w}\cap W_{\leq u}=W_{\leq w}$ which implies that
\eqref{eq-wm-5} is exact. Similarly, the case $e\neq u\leq w$ gives an exact complex
\eqref{eq-wm-5}. Therefore the interesting case to consider is $e\neq u$, $e\neq w$, 
and $u$ and $w$ are not comparable with respect to the Bruhat order.

As was pointed out to us by Axel Hultman, if there is a simple reflection $s$
such that $sw<w$ and $su<u$ or such that $ws<w$ and $us<u$, then 
the description of the differential in the BGG complex implies, by induction,
that \eqref{eq-wm-5} is exact (we will explain this in more detail
in Subsection~\ref{sec:motivation.7} below). Therefore the really interesting case is when
$u$ and $w$ neither have any common elements in their left descent sets nor 
any common elements in their right descent sets.

\subsection{$A_2$ example}\label{sec:motivation.5}
Consider $W$ of Weyl type $A_2$ with $S=\{s,t\}$. In this 
case we have $W=\{e,s,t,st,ts,w_0=sts=tst\}$.
The Gabriel quiver of the corresponding incidence algebra 
and the coefficients in the associated BGG complex look as follows:
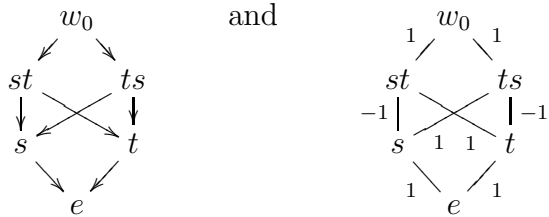
\begin{displaymath}
\xymatrix@R=3.5mm@C=1mm{
&w_0\ar[rd]\ar[ld]&\\
st\ar[rrd]\ar[d]&&ts\ar[lld]\ar[d]\\
s\ar[rd]&&t\ar[ld]\\
&e&
}
\qquad \text{ and }\qquad
\xymatrix@R=3.5mm@C=1mm{
&w_0\ar@{-}[rd]^1\ar@{-}[ld]_1&\\
st\ar@{-}[rrd]_>>>1\ar@{-}[d]_{-1}&&ts\ar@{-}[lld]^>>>1\ar@{-}[d]^{-1}\\
s\ar@{-}[rd]_1&&t\ar@{-}[ld]^1\\
&e&
}
\end{displaymath}
Here is the list of simple modules over the incidence algebra:
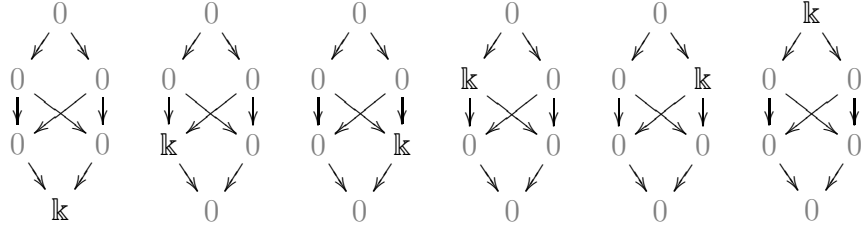
\begin{displaymath}
\xymatrix@R=3.5mm@C=1mm{
&{\color{gray}0}\ar@[gray][rd]\ar@[gray][ld]&\\
{\color{gray}0}\ar@[gray][rrd]\ar@[gray][d]&&{\color{gray}0}\ar@[gray][lld]\ar@[gray][d]\\
{\color{gray}0}\ar@[gray][rd]&&{\color{gray}0}\ar@[gray][ld]\\
&\Bbbk&
}
\quad
\xymatrix@R=3.5mm@C=1mm{
&{\color{gray}0}\ar@[gray][rd]\ar@[gray][ld]&\\
{\color{gray}0}\ar@[gray][rrd]\ar@[gray][d]&&{\color{gray}0}\ar@[gray][lld]\ar@[gray][d]\\
\Bbbk\ar@[gray][rd]&&{\color{gray}0}\ar@[gray][ld]\\
&{\color{gray}0}&
}
\quad
\xymatrix@R=3.5mm@C=1mm{
&{\color{gray}0}\ar@[gray][rd]\ar@[gray][ld]&\\
{\color{gray}0}\ar@[gray][rrd]\ar@[gray][d]&&{\color{gray}0}\ar@[gray][lld]\ar@[gray][d]\\
{\color{gray}0}\ar@[gray][rd]&&\Bbbk\ar@[gray][ld]\\
&{\color{gray}0}&
}
\quad
\xymatrix@R=3.5mm@C=1mm{
&{\color{gray}0}\ar@[gray][rd]\ar@[gray][ld]&\\
\Bbbk\ar@[gray][rrd]\ar@[gray][d]&&{\color{gray}0}\ar@[gray][lld]\ar@[gray][d]\\
{\color{gray}0}\ar@[gray][rd]&&{\color{gray}0}\ar@[gray][ld]\\
&{\color{gray}0}&
}
\quad
\xymatrix@R=3.5mm@C=1mm{
&{\color{gray}0}\ar@[gray][rd]\ar@[gray][ld]&\\
{\color{gray}0}\ar@[gray][rrd]\ar@[gray][d]&&\Bbbk\ar@[gray][lld]\ar@[gray][d]\\
{\color{gray}0}\ar@[gray][rd]&&{\color{gray}0}\ar@[gray][ld]\\
&{\color{gray}0}&
}
\quad
\xymatrix@R=3.5mm@C=1mm{
&\Bbbk\ar@[gray][rd]\ar@[gray][ld]&\\
{\color{gray}0}\ar@[gray][rrd]\ar@[gray][d]&&{\color{gray}0}\ar@[gray][lld]\ar@[gray][d]\\
{\color{gray}0}\ar@[gray][rd]&&{\color{gray}0}\ar@[gray][ld]\\
&{\color{gray}0}&
}
\end{displaymath}
And here is the list of indecomposable projective modules over the incidence algebra
(all black arrows represent the identity linear transformations):
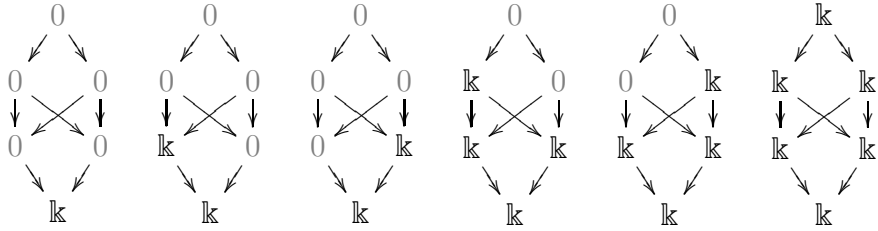
\begin{displaymath}
\xymatrix@R=3.5mm@C=1mm{
&{\color{gray}0}\ar@[gray][rd]\ar@[gray][ld]&\\
{\color{gray}0}\ar@[gray][rrd]\ar@[gray][d]&&{\color{gray}0}\ar@[gray][lld]\ar@[gray][d]\\
{\color{gray}0}\ar@[gray][rd]&&{\color{gray}0}\ar@[gray][ld]\\
&\Bbbk&
}
\quad
\xymatrix@R=3.5mm@C=1mm{
&{\color{gray}0}\ar@[gray][rd]\ar@[gray][ld]&\\
{\color{gray}0}\ar@[gray][rrd]\ar@[gray][d]&&{\color{gray}0}\ar@[gray][lld]\ar@[gray][d]\\
\Bbbk\ar[rd]&&{\color{gray}0}\ar@[gray][ld]\\
&\Bbbk&
}
\quad
\xymatrix@R=3.5mm@C=1mm{
&{\color{gray}0}\ar@[gray][rd]\ar@[gray][ld]&\\
{\color{gray}0}\ar@[gray][rrd]\ar@[gray][d]&&{\color{gray}0}\ar@[gray][lld]\ar@[gray][d]\\
{\color{gray}0}\ar@[gray][rd]&&\Bbbk\ar[ld]\\
&\Bbbk&
}
\quad
\xymatrix@R=3.5mm@C=1mm{
&{\color{gray}0}\ar@[gray][rd]\ar@[gray][ld]&\\
\Bbbk\ar[rrd]\ar[d]&&{\color{gray}0}\ar@[gray][lld]\ar@[gray][d]\\
\Bbbk\ar[rd]&&\Bbbk\ar[ld]\\
&\Bbbk&
}
\quad
\xymatrix@R=3.5mm@C=1mm{
&{\color{gray}0}\ar@[gray][rd]\ar@[gray][ld]&\\
{\color{gray}0}\ar@[gray][rrd]\ar@[gray][d]&&\Bbbk\ar[lld]\ar[d]\\
\Bbbk\ar[rd]&&\Bbbk\ar[ld]\\
&\Bbbk&
}
\quad
\xymatrix@R=3.5mm@C=1mm{
&\Bbbk\ar[rd]\ar[ld]&\\
\Bbbk\ar[rrd]\ar[d]&&\Bbbk\ar[lld]\ar[d]\\
\Bbbk\ar[rd]&&\Bbbk\ar[ld]\\
&\Bbbk&
}
\end{displaymath}
Finally, here is the list of indecomposable injective modules over the incidence algebra:
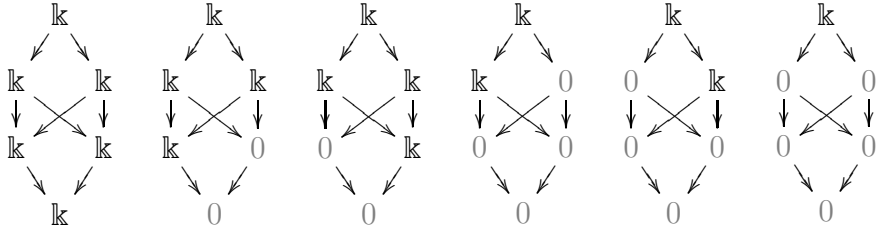
\begin{displaymath}
\xymatrix@R=3.5mm@C=1mm{
&\Bbbk\ar[rd]\ar[ld]&\\
\Bbbk\ar[rrd]\ar[d]&&\Bbbk\ar[lld]\ar[d]\\
\Bbbk\ar[rd]&&\Bbbk\ar[ld]\\
&\Bbbk&
}
\quad
\xymatrix@R=3.5mm@C=1mm{
&\Bbbk\ar[rd]\ar[ld]&\\
\Bbbk\ar@[gray][rrd]\ar[d]&&\Bbbk\ar[lld]\ar@[gray][d]\\
\Bbbk\ar@[gray][rd]&&{\color{gray}0}\ar@[gray][ld]\\
&{\color{gray}0}&
}
\quad
\xymatrix@R=3.5mm@C=1mm{
&\Bbbk\ar[rd]\ar[ld]&\\
\Bbbk\ar[rrd]\ar@[gray][d]&&\Bbbk\ar@[gray][lld]\ar[d]\\
{\color{gray}0}\ar@[gray][rd]&&\Bbbk\ar@[gray][ld]\\
&{\color{gray}0}&
}
\quad
\xymatrix@R=3.5mm@C=1mm{
&\Bbbk\ar@[gray][rd]\ar[ld]&\\
\Bbbk\ar@[gray][rrd]\ar@[gray][d]&&{\color{gray}0}\ar@[gray][lld]\ar@[gray][d]\\
{\color{gray}0}\ar@[gray][rd]&&{\color{gray}0}\ar@[gray][ld]\\
&{\color{gray}0}&
}
\quad
\xymatrix@R=3.5mm@C=1mm{
&\Bbbk\ar[rd]\ar@[gray][ld]&\\
{\color{gray}0}\ar@[gray][rrd]\ar@[gray][d]&&\Bbbk\ar@[gray][lld]\ar@[gray][d]\\
{\color{gray}0}\ar@[gray][rd]&&{\color{gray}0}\ar@[gray][ld]\\
&{\color{gray}0}&
}
\quad
\xymatrix@R=3.5mm@C=1mm{
&\Bbbk\ar@[gray][rd]\ar@[gray][ld]&\\
{\color{gray}0}\ar@[gray][rrd]\ar@[gray][d]&&{\color{gray}0}\ar@[gray][lld]\ar@[gray][d]\\
{\color{gray}0}\ar@[gray][rd]&&{\color{gray}0}\ar@[gray][ld]\\
&{\color{gray}0}&
}
\end{displaymath}

For $w=e$, we have $L_e=P_e$. Applying $\mathbb{S}$, we get the complex
$0\to I_e\to 0$ with $I_e$ at the homological position $0$. This implies
that $\mathbf{grade}(L_e)=0$.

For $w=s$, the projective resolution of $L_s$ is $0\to P_e\to P_s\to 0$. 
Applying $\mathbb{S}$, we get the complex $0\to I_e\to I_s\to 0$ with 
$I_s$ at the homological position $0$. Since the map $I_e\to I_s$
is surjective, we have that $\mathbf{grade}(L_s)=1$.
Similarly, $\mathbf{grade}(L_t)=1$.

For $w=st$, the projective resolution of $L_s$ is $0\to P_e\to P_s\oplus P_t\to P_{st}\to 0$. 
Applying $\mathbb{S}$, we get the complex 
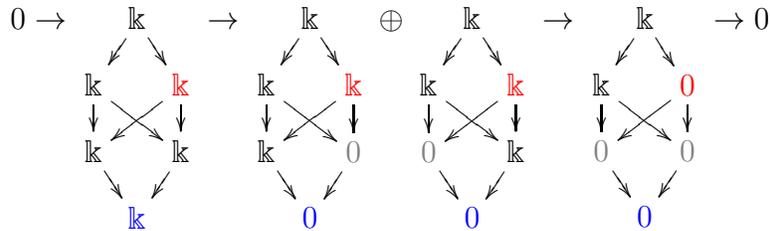
\begin{displaymath}
0\to 
\xymatrix@R=3.5mm@C=1mm{
&\Bbbk\ar[rd]\ar[ld]&\\
\Bbbk\ar[rrd]\ar[d]&&{\color{red}\Bbbk}\ar[lld]\ar[d]\\
\Bbbk\ar[rd]&&\Bbbk\ar[ld]\\
&{\color{blue}\Bbbk}&
}
\to
\xymatrix@R=3.5mm@C=1mm{
&\Bbbk\ar[rd]\ar[ld]&\\
\Bbbk\ar@[gray][rrd]\ar[d]&&{\color{red}\Bbbk}\ar[lld]\ar@[gray][d]\\
\Bbbk\ar@[gray][rd]&&{\color{gray}0}\ar@[gray][ld]\\
&{\color{blue}0}&
}
\oplus
\xymatrix@R=3.5mm@C=1mm{
&\Bbbk\ar[rd]\ar[ld]&\\
\Bbbk\ar[rrd]\ar@[gray][d]&&{\color{red}\Bbbk}\ar@[gray][lld]\ar[d]\\
{\color{gray}0}\ar@[gray][rd]&&\Bbbk\ar@[gray][ld]\\
&{\color{blue}0}&
}
\to
\xymatrix@R=3.5mm@C=1mm{
&\Bbbk\ar@[gray][rd]\ar[ld]&\\
\Bbbk\ar@[gray][rrd]\ar@[gray][d]&&{\color{red}0}\ar@[gray][lld]\ar@[gray][d]\\
{\color{gray}0}\ar@[gray][rd]&&{\color{gray}0}\ar@[gray][ld]\\
&{\color{blue}0}&
}
\to 0
\end{displaymath}
with $I_{st}$ at the homological position $0$. The restriction of this
complex to the vertex ${\color{red}ts}$ (shown in {\color{red}red}) gives the complex
\begin{displaymath}
0\to \Bbbk\to \Bbbk\oplus \Bbbk\to 0\to 0, 
\end{displaymath}
supported at $W_{\leq st}\cap W_{\leq ts}=\{e,s,t\}$. It has non-zero homology 
at position $-1$. The only other restriction to a vertex resulting in a
non-zero homology is that to {\color{blue}$e$}
(shown in {\color{blue}blue}) which 
gives the complex
\begin{displaymath}
0\to \Bbbk\to 0\oplus 0\to 0\to 0, 
\end{displaymath}
supported at $W_{\leq st}\cap W_{\leq e}=\{e\}$. It has non-zero homology 
at position $-2$.  By taking the minimum of $1$ and $2$, we obtain $\mathbf{grade}(L_{st})=1$.
Similarly, $\mathbf{grade}(L_{ts})=1$.

For $w=w_0$, the projective resolution of $L_{w_0}$ 
is $0\to P_e\to P_s\oplus P_t\to P_{st}\oplus P_{ts}\to P_{w_0}\to 0$. 
Applying $\mathbb{S}$, we get the complex 
\begin{displaymath}
0\to 
\xymatrix@R=3.5mm@C=1mm{
&\Bbbk\ar[rd]\ar[ld]&\\
\Bbbk\ar[rrd]\ar[d]&&\Bbbk\ar[lld]\ar[d]\\
\Bbbk\ar[rd]&&\Bbbk\ar[ld]\\
&{\color{blue}\Bbbk}&
}
\to
\xymatrix@R=3.5mm@C=1mm{
&\Bbbk\ar[rd]\ar[ld]&\\
\Bbbk\ar@[gray][rrd]\ar[d]&&\Bbbk\ar[lld]\ar@[gray][d]\\
\Bbbk\ar@[gray][rd]&&{\color{gray}0}\ar@[gray][ld]\\
&{\color{blue}0}&
}
\oplus
\xymatrix@R=3.5mm@C=1mm{
&\Bbbk\ar[rd]\ar[ld]&\\
\Bbbk\ar[rrd]\ar@[gray][d]&&\Bbbk\ar@[gray][lld]\ar[d]\\
{\color{gray}0}\ar@[gray][rd]&&\Bbbk\ar@[gray][ld]\\
&{\color{blue}0}&
}
\to
\xymatrix@R=3.5mm@C=1mm{
&\Bbbk\ar@[gray][rd]\ar[ld]&\\
\Bbbk\ar@[gray][rrd]\ar@[gray][d]&&{\color{gray}0}\ar@[gray][lld]\ar@[gray][d]\\
{\color{gray}0}\ar@[gray][rd]&&{\color{gray}0}\ar@[gray][ld]\\
&{\color{blue}0}&
}
\oplus 
\xymatrix@R=3.5mm@C=1mm{
&\Bbbk\ar[rd]\ar@[gray][ld]&\\
{\color{gray}0}\ar@[gray][rrd]\ar@[gray][d]&&\Bbbk\ar@[gray][lld]\ar@[gray][d]\\
{\color{gray}0}\ar@[gray][rd]&&{\color{gray}0}\ar@[gray][ld]\\
&{\color{blue}0}&
}
\to
\xymatrix@R=3.5mm@C=1mm{
&\Bbbk\ar@[gray][rd]\ar@[gray][ld]&\\
{\color{gray}0}\ar@[gray][rrd]\ar@[gray][d]&&{\color{gray}0}\ar@[gray][lld]\ar@[gray][d]\\
{\color{gray}0}\ar@[gray][rd]&&{\color{gray}0}\ar@[gray][ld]\\
&{\color{blue}0}&
}
\to 0
\end{displaymath}
with $I_{w_0}$ at the homological position $0$. The restriction of this
complex to the vertex {\color{blue}$e$} (shown in {\color{blue}blue}) gives the complex
\begin{displaymath}
0\to \Bbbk\to 0\oplus 0\to 0\oplus 0\to 0\to 0, 
\end{displaymath}
supported at $W_{\leq w_0}\cap W_{\leq e}=\{e\}$. It has non-zero homology 
at position $-3$ and hence $\mathbf{grade}(L_{st})=3$.

To sum up, here are the values of the grade function in type $A_2$:
\begin{displaymath}
\begin{array}{c||c|c|c|c|c|c}
w&e&s&t&st&ts&w_0\\ \hline\hline
\mathbf{grade}(L_w)&0&1&1&1&1&3
\end{array}
\end{displaymath}

One could observe that these values coincide with the values of 
Lusztig's $\mathbf{a}$-function from \cite{lusztig} in this case.
We will come back to this observation in more detail at the end of the paper.

\subsection{Connection to Auslander regularity}\label{sec:motivation.6}

Our interest in the grades of simple modules stems from the theory of
Auslander regular algebras, see \cite{iyama_marczinzik}. 
It is shown in \cite{iyama_marczinzik}
that the incidence algebra of a lattice is Auslander regular if and only
if the lattice is distributive. The poset $(W,\leq)$ is not a lattice,
in general, and, outside type $A$ there are reasons why, generically,
the incidence algebra of $(W,\leq)$ is not Auslander regular. 
In type $A$ this question is still open 
(as was mentioned by Rene Marczinzik at a seminar talk in Uppsala
in March 2021) and the research presented in
this paper is originally motivated by that problem. 
Grades of simple modules are essential homological invariants
for this theory and they behave especially nicely for some
Auslander regular algebras, see \cite{iyama_marczinzik} for details.

\subsection{Matchings}\label{sec:motivation.7}

Let $Q$ be a convex subset of $W$ in the sense that $x,y\in Q$
with $x\leq y$ implies $[x,y]\subseteq Q$. 
Then we can restrict the BGG complex to its part supported
at $Q$; i.e., to the linear span of all vectors indexed by the elements in $Q$.
Let us denote this complex by $V_\bullet^Q$.

\begin{lemma}\label{lem-wm-31}
Assume that the poset $(Q,\leq)$ admits a filtration
\begin{equation}\label{eq-wm-311-1}
\emptyset=Q_0\subset Q_1\subset\dots\subset Q_l=Q 
\end{equation}
by coideals such that each $Q_i\setminus Q_{i-1}=\{x_i,y_i\}$, 
where $x_i<y_i$ and $\ell(x_i)=\ell(y_i)-1$. Then $V_\bullet^Q$ is exact. 
\end{lemma}

\begin{proof}
The filtration \eqref{eq-wm-311-1} gives rise to a filtration of 
$V_\bullet^Q$ by subcomplexes with subquotients of the form
\begin{equation}\label{eq-wm-312}
0\to \mathbb{C}\langle y_i\rangle \to \mathbb{C}\langle x_i\rangle\to 0.
\end{equation}
From the definition of the BGG complex we see that the map 
$\mathbb{C}\langle y_i\rangle \to \mathbb{C}\langle x_i\rangle$
is non-zero and therefore \eqref{eq-wm-312} is homotopic to zero.
The claim follows.
\end{proof}

We will call the decomposition of $Q$ into the subsets $\{x_i,y_i\}$ 
given by Lemma~\ref{lem-wm-31} a {\em perfect matching}. 
The already mentioned observation by Axel Hultman
was that existence of a simple reflection $s$ such that 
$sv<v$ and $sw<w$ obviously implies that $B(v)\cap B(w)$ has a perfect matching
by the pairs $\{x,sx\}$. In particular, $V_\bullet^{B(v)\cap B(w)}$ is exact
in this case. Similarly in the case of right descents; i.e., for $vs<v$ and $ws<w$.

\begin{lemma}\label{lem-wm-32}
Assume that the poset $(Q,\leq)$ admits a filtration
$$\emptyset=Q_0\subset Q_1\subset\dots\subset Q_l=Q $$
by coideals such that each $Q_i\setminus Q_{i-1}=\{x_i,y_i\}$, 
where $x_i<y_i$ and $\ell(x_i)=\ell(y_i)-1$, with the exception
of one $i$ for which we get  a singleton $z$. Then $V_\bullet^Q$ 
has exactly one non-zero homology, namely in the homological 
position $\ell(z)$ and this homology is one-dimensional. 
\end{lemma}

\begin{proof}
Similarly to the proof of  Lemma~\ref{lem-wm-31}, all matched pairs will
give subquotient complexes that are homotopic to zero, so the one-dimensional 
homology will be concentrated at the unique unmatched singleton. 
\end{proof}

We will call a decomposition of $Q$ given by Lemma~\ref{lem-wm-32}
an {\em almost perfect matching}. 

\section{Matchings in intersections}\label{sec:matchings}

\subsection{Preliminaries}\label{sec:matchings.1}

Throughout this section, let $v \in \mf{S}_n$ be a boolean permutation. 
As discussed in Section~\ref{sec:motivation}, we want to find a permutation 
$w \in \mf{S}_n$ such that there is an almost perfect matching 
of the elements of $B(v) \cap B(w)$, in which the singleton element 
is of highest possible rank. Such a $w$ will be called an 
\emph{optimal partner} for $v$, and that highest possible rank 
is the \emph{optimal rank} of $v$, denoted $\optimalrank(v)$. 
Note that there always exists an intersection $B(v) \cap B(w)$ that 
can be almost perfectly matched, because we could let $w$ be the 
identity permutation. So $\optimalrank(v)$ always exists.

We will also show in Proposition~\ref{prop:bounding optimal rank} that, for any $w$, the poset $B(v) \cap B(w)$
has either a perfect matching or an almost perfect matching.
We begin by formalizing a fact mentioned in the previous section.

\begin{lemma}\label{lem:optimal partner can't be larger than v}
An optimal partner for $v\neq e$ can never be greater than or equal to $v$ in the Bruhat order.
\end{lemma}

\begin{proof}
A principal order ideal containing more than one element always has zero homology. If $v \in B(w)$, then $B(v) \cap B(w) = B(v)$, which would mean that $w$ is not an optimal partner for $v$.
\end{proof}

As an immediate corollary, we record a property of the optimal rank function.

\begin{corollary}\label{cor:optimal rank < l(v)}
For any $v$ that is not the identity permutation, $\optimalrank(v) < \ell(v)$.
\end{corollary}

Recall Lemma~\ref{lem:commuting generators make product poi}, which gives a way to 
decompose order ideals based on their commuting subsets of support. 
Suppose, for the moment, that $\supp(v) = X \sqcup Y$ as in that lemma. 
For $\rw{s} \in R(w)$, define $\rw{s_X}$ and $\rw{s_Y}$ accordingly, and we 
then observe that $B(v) \cap B(w)$ would be isomorphic to 
$(B(\rw{s_X}) \cap B(w)) \times (B(\rw{s_Y}) \cap B(w))$. Therefore, if there is 
a perfect  matching of $B(\rw{s_X}) \cap B(w)$, then we can construct a 
perfect matching of $B(v) \cap B(w)$, as a product of the matching filtration of
$B(\rw{s_X}) \cap B(w)$ with any filtration of $B(\rw{s_Y}) \cap B(w)$
by coideals with singleton subquotients. Therefore, we can often assume that $\supp(v) = [1,n-1]$.

\subsection{Runs}\label{sec:matchings.2}

We begin with an influential special case.

\begin{definition}
A sequence of consecutive integers that is either increasing or decreasing is a \emph{run}.
\end{definition}

\begin{lemma}\label{lem:runs}
Suppose that the run $\rw{a(a+1)\cdots (a+b)}$ (resp., $\rw{(a+b)\cdots (a+1)a}$) is a reduced word for $v$. Then an optimal partner for $v$ is
$$\rw{(a+1)(a+2) \cdots (a+b) a (a+1) \cdots (a+b-1)}$$
(resp., $\rw{(a+b-1)\cdots (a+1)a (a+b)\cdots (a+2)(a+1)}$) if $b>0$, and an optimal partner is the identity if $b=0$.
\end{lemma}

\begin{proof}
It is sufficient to consider $a = 1$.

The case $b = 0$ follows from Lemma~\ref{lem:optimal partner can't be larger than v}.

Now suppose that $b > 0$, and, without loss of generality, that $v = \rw{12\cdots (1+b)}$. Consider the permutation $w := \rw{23\cdots (1+b) 12\cdots b}$. Using Theorem~\ref{thm:subword property}, we see that
$$B(v) \cap B(w) = B(v) \setminus \{v\}.$$
The set $B(v)$ is a union of left cosets for the parabolic subgroup generated by $\sigma_1$. 
We can consider the restriction of the Bruhat order to the set of minimal coset representatives.
Any filtration of the latter poset by coideals with singleton subquotients extends in the obvious way to
a filtration of $B(v)$ by coideals with subquotients being exactly the cosets. This gives a
perfect matching for $B(v)$.  Removing $v$, which must belong to the first coset in the
filtration, we obtain an almost perfect matching for $B(v) \setminus \{v\}$. The unmatched
element $\sigma_1v$ has length $\ell(v) - 1$, and hence this must be the optimal rank, 
by Corollary~\ref{cor:optimal rank < l(v)}. Thus $w$ must be an optimal partner for $v$.
\end{proof}

We demonstrate Lemma~\ref{lem:runs} with an example.

\begin{example}
Consider $v = 51234 = \rw{4321}$. Then $w = 45123 = \rw{321432}$ is an optimal partner for $v$. The poset $B(v) \cap B(w)$ is shown in Figure~\ref{fig:51234 example}, with thick lines indicating the almost perfect matching from Lemma~\ref{lem:runs}.
\begin{figure}[htbp]
\begin{tikzpicture}[scale=.8]
\fill (0,0) coordinate (e) circle (2pt) node[below] {$e$};
\fill (-4.5,2) circle (2pt) coordinate (1) node[below left] {$\rw{1}$};
\fill (-1.5,2) circle (2pt) coordinate (2) node[below left] {$\rw{2}$};
\fill (1.5,2) circle (2pt) coordinate (3) node[below right] {$\rw{3}$};
\fill (4.5,2) circle (2pt) coordinate (4) node[below right] {$\rw{4}$};
\fill (-7.5,4) circle (2pt) coordinate (21) node[left] {$\rw{21}$};
\fill (-4.5,4) circle (2pt) coordinate (31) node[left] {$\rw{31}$};
\fill (-1.5,4) circle (2pt) coordinate (41) node[left] {$\rw{41}$};
\fill (1.5,4) circle (2pt) coordinate (32) node[right] {$\rw{32}$};
\fill (4.5,4) circle (2pt) coordinate (42) node[right] {$\rw{42}$};
\fill (7.5,4) circle (2pt) coordinate (43) node[right] {$\rw{43}$};
\fill (-4.5,6) circle (2pt) coordinate (321) node[above] {$\rw{321}$};
\fill (-1.5,6) circle (2pt) coordinate (421) node[above] {$\rw{421}$};
\fill (1.5,6) circle (2pt) coordinate (431) node[above] {$\rw{431}$};
\fill (4.5,6) circle (2pt) coordinate (432) node[above] {$\rw{432}$};
\foreach \x in {1,2,3,4} {\draw (e) -- (\x);}
\foreach \x in {21,31,41} {\draw (1) -- (\x);}
\foreach \x in {21,32,42} {\draw (2) -- (\x);}
\foreach \x in {31,32,43} {\draw (3) -- (\x);}
\foreach \x in {41,42,43} {\draw (4) -- (\x);}
\foreach \x in {21,31,32} {\draw (321) -- (\x);}
\foreach \x in {21,41,42} {\draw (421) -- (\x);}
\foreach \x in {31,41,43} {\draw (431) -- (\x);}
\foreach \x in {32,42,43} {\draw (432) -- (\x);}
\draw[line width=3pt] (e) -- (1);
\draw[line width=3pt] (2) -- (21);
\draw[line width=3pt] (3) -- (31);
\draw[line width=3pt] (4) -- (41);
\draw[line width=3pt] (32) -- (321);
\draw[line width=3pt] (42) -- (421);
\draw[line width=3pt] (43) -- (431);
\end{tikzpicture}
\caption{A matching of the elements of $B(\rw{432}) \cap B(\rw{321432})$, in which only $\rw{432}$ is unmatched.}\label{fig:51234 example}
\end{figure}
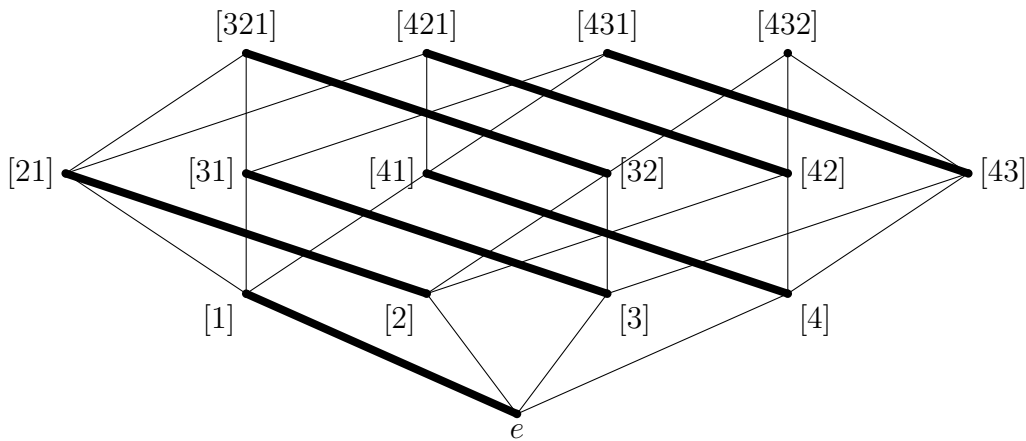
\end{example}

The cases $b = 0$ and $b > 0$ in the previous lemma share an important property.

\begin{corollary}\label{cor:optimal rank of runs}
If $v$ is boolean and has a reduced word that is a run, then $\optimalrank(v) = \ell(v) - 1$.
\end{corollary}

Lemma~\ref{lem:runs} seems to consider a very particular situation. However, because we are assuming that $v$ is boolean, we can actually view any $\rw{s} \in R(v)$ as a product of disjoint runs.

\begin{example}\label{ex:runs in 24153}
$R(24153) = \{\rw{1324},\rw{3124},\rw{1342},\rw{3142},\rw{3412}\}$. The first element of that set can be viewed as the concatenation of three runs: $1\cdot32\cdot4$, whereas the last element can be viewed as the concatenation of two runs: $34\cdot12$. Of course, we could also write $1\cdot3\cdot2\cdot4$ and so on, but this inefficiency is not helpful, as described below.
\end{example}

\subsection{Optimal ranks}\label{sec:matchings.3}

Corollary~\ref{cor:optimal rank of runs} suggests that the optimal rank of $v$ might be related to the fewest number of runs needed to form a reduced word for $v$, and indeed that is the main result of this section.

\begin{definition}
Fix a boolean permutation $v$. Let $\run(v)$ be the fewest number of runs needed in any concatenation forming a reduced word for $R(v)$. A reduced word $\rw{s} \in R(w)$ that can be written as the concatenation of $\run(v)$ runs is an \emph{optimal run word} for $v$.
\end{definition}

Recalling Example~\ref{ex:runs in 24153}, we have that $\run(24153) = 2$, and $\rw{3412}$ is an optimal run word for $24153$.

\begin{example}\label{ex:wm-commuting}
If $v$ is boolean and, additionally, a product of pairwise commuting simple reflections, then 
$\run(v)=\ell(v)$ and any reduced word of $v$ is an optimal run word. In this case
it is also easy to prove that $\optimalrank(v)=0$.  Indeed, for any $w$, we obviously have
$B(v)\cap B(w)=B(u)$, where $u$ is the product of simple reflections in 
$\mathrm{supp}(v)\cap \mathrm{supp}(w)$. Hence, if $u\neq e$, the set $B(u)$ has a perfect matching.
If $u=e$, then $B(u)$ has an almost perfect matching with singleton of length $0$. 
\end{example}

The main result of this section will determine the optimal rank and an optimal partner for any boolean permutation. Before doing so, we will give an upper bound on the optimal rank, using a handy lemma. 

\begin{lemma}\label{lem:deleting a letter and slimming}
Fix a reduced word $\rw{s} = \rw{s_1\cdots s_l} \in R(w)$ and $i \in [1,l]$. Set $\hat{s} := s_1\cdots \hat{s_i} \cdots s_l$. 
There exists a unique maximal $w' \in B(w)$ having a reduced word that is a subword of $\hat{s}$. In other words, the permutations whose reduced words are subwords of $\hat{s}$ form a principal order ideal. 
Moreover, the boolean elements of $B(w)$ with a reduced word that can be written as a subword of $\hat{s}$ are exactly the boolean elements of $B(w')$.
\end{lemma}

\begin{proof}
It is easy to check that the lemma holds for permutations of small lengths. In particular, if $\ell(w) = 1$ then $w' = e$. Suppose now that the result is true for all permutations $u$ with $\ell(u) < \ell(w)$. 

If $i = l$, then $\hat{s}$ is necessarily reduced. Therefore we have $w' = \rw{\hat{s}}$, and the property for boolean elements follows immediately.

Now suppose that $i < l$. 
Define the string $t := s_1\cdots s_{l-1}$ and set $u := \rw{t}$, noting that $\ell(u) = l-1 < \ell(w)$. 
Set $\hat{t} := s_1 \cdots \hat{s_i} \cdots s_{l-1}$, and let $u' \in B(u)$ be the permutation produced by the inductive hypothesis. 
Define a permutation $\overline{w}$ as follows:
\begin{equation}\label{eqn:slimmed permutation}
\overline{w}:= \begin{cases}
u' \s_{s_l} 
& \text{if } \ell(u'\s_{s_l}) > \ell(u'), \text{ and}\\
u' & \text{if } \ell(u'\s_{s_l}) < \ell(u').
\end{cases}
\end{equation}
By definition, this $\overline{w}$ has a reduced word that is a subword of $\hat{s}$. It remains to establish that $\overline{w}$ is maximal with this property and that $B(\overline{w})\subset B(w)$ contains the necessary boolean elements.

Let $x$ be maximal among permutations having reduced words that are subwords of $\hat{s}$. Fix $\rw{q} \in R(x)$ such that $q$ is a subword of $\hat{s}$ and, if possible, the rightmost letter of $q$ is not $s_l$.

If $q$ does not end with $s_l$, then $q$ is a subword of $\hat{t}$. Maximality of $u'$ means that $x =  u'$. Because this $x$ is maximal, we must have that $\ell(u' \s_{s_l}) < \ell(u')$, and so $x = \overline{w}$.

On the other hand, suppose that it is impossible to write $q$ in this way, and so $q = q' s_l$ for a subword $q'$ of $\hat{t}$. This means that $\rw{q'} \le u'$ in the Bruhat order, and $\s_{s_l}$ is not a right descent of $u'$. But then $\overline{w} = u'\s_{s_l}$, and so we must have $\rw{q'} = u'$ by maximality. Therefore $x = \rw{q's_l} = \overline{w}$.

Therefore, $w' := \overline{w}$ is the desired permutation.

We now use the inductive hypothesis and Equation~\eqref{eqn:slimmed permutation} to prove the second half of the lemma. Let $v$ be a boolean permutation with a reduced word that is a subword of $\hat{s}$. If $v$ actually has a reduced word that is a subword of $\hat{t}$, then $v \in B(u') \subseteq B(w')$ by the inductive hypothesis. If $v$ has no such reduced word, then it must be the case that $w' = u'\s_{s_l}$. Fix $\rw{r} \in R(v)$ such that $r$ is a subword of $\hat{s}$. Then $r$ is a concatenation of $r' \subseteq \hat{t}$ and $s_l$. By the inductive hypothesis, $\rw{r'} \in B(u')$ and hence $v = \rw{r' \cdot s_l} \in B(u'\s_{s_l}) = B(w')$.
\end{proof}

We demonstrate Lemma~\ref{lem:deleting a letter and slimming} with an example.

\begin{example}
Let $w = 4321$ and $\rw{s} = \rw{321232} \in R(w)$, with $i = 3$. Thus $\hat{s} = 32232$. The proof of Lemma~\ref{lem:deleting a letter and slimming} builds the permutation $w'$ inductively from permutations $u'_1, \ldots, u'_5 = w'$ as follows.
$${\renewcommand{\arraystretch}{1.5}\begin{tabular}{c||c|c|c|c|c}
\ \ $j$ \ \ & $1$ & $2$ & $3$ & $4$ & $5$\\
\hline
\hline
$u'_j$ & $\rw{3}$ & $\rw{32}$ & $\rw{32}$ & $\rw{323}$ & $\rw{323}$ 
\end{tabular}}$$
The ideal $B(w)$ has thirteen boolean elements, five of which can be formed from subwords of $\hat{s}$:
\begin{equation}\label{eqn:slimming example}
\{e, \rw{2}\!, \rw{3}\!, \rw{23}\!, \rw{32}\}.
\end{equation}
The boolean elements of $B(\rw{232})$ are exactly the five permutations listed in~\eqref{eqn:slimming example}.
\end{example}

The ability to delete letters from a reduced word without losing particular elements from its principal order ideal is important for the inductive step in the following proposition.

\begin{proposition}\label{prop:bounding optimal rank}
For any boolean permutation $v$, we have $\optimalrank(v) \le \ell(v) - \run(v)$.
\end{proposition}

\begin{proof}
We prove this result by induction on $\ell(v)$, and it is easy to verify that the proposition holds for $\ell(v) \le 2$. In fact, we also know that the proposition holds whenever $\run(v) = 1$, by Corollary~\ref{cor:optimal rank of runs}. Suppose, inductively, that for any boolean $u$ with $\ell(u) < \ell(v)$, any intersection $B(u) \cap B(y)$ either has a perfect matching or has an almost perfect matching with a single unmatched element of rank at most $\ell(u) - \run(u)$.

Fix $\rw{s} \in R(v)$ and an arbitrary permutation $w$. We want to show that $B(v) \cap B(w)$ either has a perfect matching or has an almost perfect matching with a single unmatched element of rank at most $\ell(v) - \run(v)$. Set $m := \max (\supp(v) \cap \supp(w))$. For $z \in B(v) \cap B(w)$ with $m \not\in \supp(z)$, we match $z \longleftrightarrow z \addletter \s_m$ whenever $z \addletter \s_m$ exists. Similarly to the proof of Lemma~\ref{lem:runs}, this matching is
inherited from the filtration with respect to the Bruhat order  on the set of shortest coset
representatives for the parabolic subgroup generated by $\s_m$.
If $z \addletter \s_m$ always exists, then this matches all elements of $B(v) \cap B(w)$ and we are done.

Now suppose that this matching does not account for all elements of $B(v) \cap B(w)$, and let $X$ be the set of as-yet-unmatched elements. None of these elements have $m$ in their supports. Moreover, because it is impossible to introduce $m$ anywhere in their reduced words remaining inside $B(v) \cap B(w)$, they must all have (at least) $m-1$. Let $m' < m$ be maximal such that the subword $\rw{s}_{[m',m]}$ does not appear in any element of $R(w)$. (Note that $[m',m] \subseteq \supp(v) \cap \supp(w)$, by maximality of $m'$.) Then
$$q := \rw{s_{[m',m-1]}} \in X,$$
and any $x \in X$ must include $\rw{s_{[m',m]}}$ in its reduced words, so $x$ is greater than or equal to $q$ in the Bruhat order. Thus $X$ is a filter; that is, $X$ is the principal coideal of  $B(v) \cap B(w)$ generated by $q$. This is very good news
as it now allows us to construct the matching we are looking for, inductively. Note that
$$\ell(q) = m-1-m'+1 = m-m'.$$

Let $v':=\rw{s_{[1,m'-1]}}$ be the permutation obtained by deleting the letters $[m',m]$ from reduced words for $v$. Without loss of generality, suppose that $m'-1$ does not appear to the right of $m'$ in elements of $R(v)$. The permutation $v$ was boolean, so $v'$ is boolean and
$$\ell(v') = \ell(v) - (m-m' + 1).$$
Take a reduced word for $w$, look for all substrings that match reduced words for $q$, and mark the rightmost copy of $m'$ used in any of these. Working iteratively, delete each $m'-1$ that appears to the right of the marked $m'$, using Lemma~\ref{lem:deleting a letter and slimming} to produce a reduced word after each deletion. When there are no more copies of $m'-1$ appearing to the right of the marked $m'$, write $w'$ for the permutation described by the resulting reduced word.

We defined $v'$ and $w'$ for the purpose of our inductive argument: the filter $X$ is isomorphic to $B(v') \cap B(w')$, with $\rw{x_1 \cdots x_l} \in B(v') \cap B(w')$ corresponding to $q \addletter \s_{x_1} \addletter \cdots \addletter \s_{x_l}$. By the inductive hypothesis, $\optimalrank(v') \le \ell(v') - \run(v')$, so the poset $B(v') \cap B(w')$ has either a perfect matching or an almost perfect matching with an unmatched element of rank at most $\ell(v') - \run(v')$.

By definition of $m'$ and $q$, the sequences $\rw{s_{[m'+1,m]}}$ and $\rw{s_{[m',m-1]}}$ both appear in reduced words for $w$, but $\rw{s_{[m',m]}}$ does not. The only way for this to happen is for $[m',m]$ to be a run in $\rw{s}$ (and for reduced words for $w$ to contain a subsequence as in Lemma~\ref{lem:runs}). Thus
$$\run(v') \ge \run(v) - 1.$$
By the inductive hypothesis, the intersection $B(v') \cap B(w')$ either has a perfect matching or it has an almost perfect matching with an unmatched element of rank at most $\ell(v') - \run(v')$. Transfer this matching onto $X \subset B(v) \cap B(w)$. This, together with the initial matching $z \longleftrightarrow z \addletter \s_m$, produces either a perfect matching of $B(v) \cap B(w)$ or an almost perfect matching whose single unmatched element has rank at most
\begin{align*}
\ell(q) + \optimalrank(v') &\le \ell(q) + \ell(v') - \run(v')\\
&\le (m-m') + \ell(v) - (m-m'+1) - (\run(v) - 1)\\
&=\ell(v) - \run(v),
\end{align*}
completing the proof.
\end{proof}

We demonstrate how this bound can work, along with the inductive argument.

\begin{example}
We proceed with two examples, using the notation of Proposition~\ref{prop:bounding optimal rank}.
\begin{enumerate}\renewcommand{\labelenumi}{(\alph{enumi})}
\item Let $v = 2341 = \rw{123}$ and $w = 4123 = \rw{321}$, so $m = 3$ and $m' = 2$. We match
\begin{eqnarray*}
\emptyset &\longleftrightarrow& \rw{3}\\
\rw{1} &\longleftrightarrow& \rw{13}
\end{eqnarray*}
and $X = \{\rw{2}\}$. Then $q := \rw{123_{[2,2]}} = \rw{2}$ and $v' := \rw{123_{[1,1]}} = \rw{1}$, and $\optimalrank(v') = 0$. Lemma~\ref{lem:deleting a letter and slimming} gives $w' = \rw{32}$, and hence $B(v') \cap B(w') = \emptyset \cong X$. This example is depicted in Figure~\ref{fig:bounding optimal rank of 123 ex}.
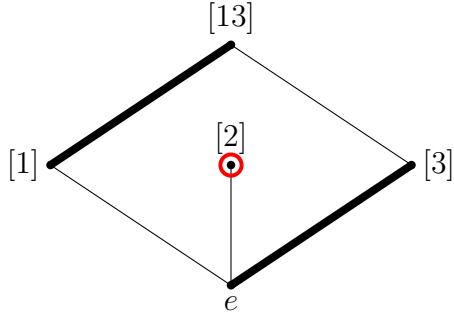
\begin{figure}[htbp]
\begin{tikzpicture}[scale=.8]
\fill (0,0) coordinate (e) circle (2pt) node[below] {$e$};
\fill (-3,2) coordinate (1) circle (2pt) node[left] {$\rw{1}$};
\fill (0,2) coordinate (2) circle (2pt) node[above] {$\rw{2}$};
\fill (3,2) coordinate (3) circle (2pt) node[right] {$\rw{3}$};
\fill (0,4) coordinate (13) circle (2pt) node[above] {$\rw{13}$};
\draw (e) -- (1);
\draw (e) -- (2);
\draw (3) -- (13);
\draw[line width=3pt] (e) -- (3);
\draw[line width=3pt] (1) -- (13);
\draw[line width=1.5pt, red] (2) circle (5pt);
\end{tikzpicture}
\caption{The intersection $B(\rw{123}) \cap B(\rw{321})$, and the matching described by Proposition~\ref{prop:bounding optimal rank}. The sole element of the filter $X$, described in the proof of that proposition, is circled in red. The unmatched element in this almost perfect matching has rank $1 < \ell(\rw{123}) - \run(\rw{123})$.}\label{fig:bounding optimal rank of 123 ex}
\end{figure}
\item Let $v = 314562 = \rw{23451}$ and $w = 235614 = \rw{412534}$, so $m = 5$ and $m' = 3$. Then we match
\begin{eqnarray*}
\emptyset &\longleftrightarrow& \rw{5}\\
\rw{1} &\longleftrightarrow& \rw{15}\\
\rw{2} &\longleftrightarrow& \rw{25}\\
\rw{3} &\longleftrightarrow& \rw{35}\\
\rw{4} &\longleftrightarrow& \rw{45}\\
\rw{13} &\longleftrightarrow& \rw{135}\\
\rw{14} &\longleftrightarrow& \rw{145}\\
\rw{23} &\longleftrightarrow& \rw{235}\\
\rw{24} &\longleftrightarrow& \rw{245}
\end{eqnarray*}
and $X = \{\rw{34}, \rw{134}, \rw{234}\}$. Then $q := \rw{23451_{[3,4]}} = \rw{34}$ and $v' := \rw{23451_{[1,2]}} = \rw{21}$, and $\optimalrank(v') = 1$. Lemma~\ref{lem:deleting a letter and slimming} gives $w' = w$, and hence $B(v') \cap B(w') = \{\emptyset,\rw{1},\rw{2}\} \cong X$. This example is depicted in Figure~\ref{fig:bounding optimal rank of 23451 ex}.
\begin{figure}[htbp]
\begin{tikzpicture}[scale=.8]
\fill (0,0) coordinate (e) circle (2pt) node[below] {$e$};
\fill (-6,2) coordinate (1) circle (2pt) node[left] {$\rw{1}$};
\fill (-3,2) coordinate (2) circle (2pt) node[left] {$\rw{2}$};
\fill (0,2) coordinate (3) circle (2pt) node[below left] {$\rw{3}$};
\fill (3,2) coordinate (4) circle (2pt) node[below right] {$\rw{4}$};
\fill (6,2) coordinate (5) circle (2pt) node[right] {$\rw{5}$};
\fill (-8,4) coordinate (13) circle (2pt) node[left] {$\rw{13}$};
\fill (-6,4) coordinate (14) circle (2pt) node[left] {$\rw{14}$};
\fill (-4,4) coordinate (15) circle (2pt) node[left] {$\rw{15}$};
\fill (-2,4) coordinate (23) circle (2pt) node[left] {$\rw{23}$};
\fill (0,4) coordinate (24) circle (2pt) node[below] {$\rw{24}$};
\fill (2,4) coordinate (25) circle (2pt) node[below] {$\rw{25}$};
\fill (4,4) coordinate (34) circle (2pt) node[right] {$\rw{34}$};
\fill (6,4) coordinate (35) circle (2pt) node[right] {$\rw{35}$};
\fill (8,4) coordinate (45) circle (2pt) node[right] {$\rw{45}$};
\fill (-5,6) coordinate (134) circle (2pt) node[above] {$\rw{134}$};
\fill (-3,6) coordinate (135) circle (2pt) node[above] {$\rw{135}$};
\fill (-1,6) coordinate (145) circle (2pt) node[above] {$\rw{145}$};
\fill (1,6) coordinate (234) circle (2pt) node[above] {$\rw{234}$};
\fill (3,6) coordinate (235) circle (2pt) node[above] {$\rw{235}$};
\fill (5,6) coordinate (245) circle (2pt) node[above] {$\rw{245}$};
\foreach \x in {1,2,3,4,5} {\draw (e) -- (\x);}
\foreach \x in {1,3,134,135} {\draw (13) -- (\x);}
\foreach \x in {1,4,134,145} {\draw (14) -- (\x);}
\foreach \x in {1,5,135,134} {\draw (15) -- (\x);}
\foreach \x in {2,3,234,235} {\draw (23) -- (\x);}
\foreach \x in {2,4,234,245} {\draw (24) -- (\x);}
\foreach \x in {2,5,235,245} {\draw (25) -- (\x);}
\foreach \x in {3,4,134,234} {\draw (34) -- (\x);}
\foreach \x in {3,5,135,235} {\draw (35) -- (\x);}
\foreach \x in {4,5,145,245} {\draw (45) -- (\x);}
\draw[line width=3pt] (e) -- (5);
\draw[line width=3pt] (1) -- (15);
\draw[line width=3pt] (2) -- (25);
\draw[line width=3pt] (3) -- (35);
\draw[line width=3pt] (4) -- (45);
\draw[line width=3pt] (13) -- (135);
\draw[line width=3pt] (14) -- (145);
\draw[line width=3pt] (23) -- (235);
\draw[line width=3pt] (24) -- (245);
\draw[red, line width=1pt] (134) -- (34);
\draw[red, line width=3pt] (34) -- (234);
\foreach \x in {34,134,234} {\fill (\x) circle (2pt); \draw[line width=1.5pt, red] (\x) circle (5pt);}
\end{tikzpicture}
\caption{The intersection $B(\rw{23451}) \cap B(\rw{412534})$, and the matching described by Proposition~\ref{prop:bounding optimal rank}. The filter $X$, described in the proof of that proposition, and its matching are marked in red. The unmatched element in this almost perfect matching has rank $3 = \ell(\rw{23451}) - \run(\rw{23451})$.}\label{fig:bounding optimal rank of 23451 ex}
\end{figure}
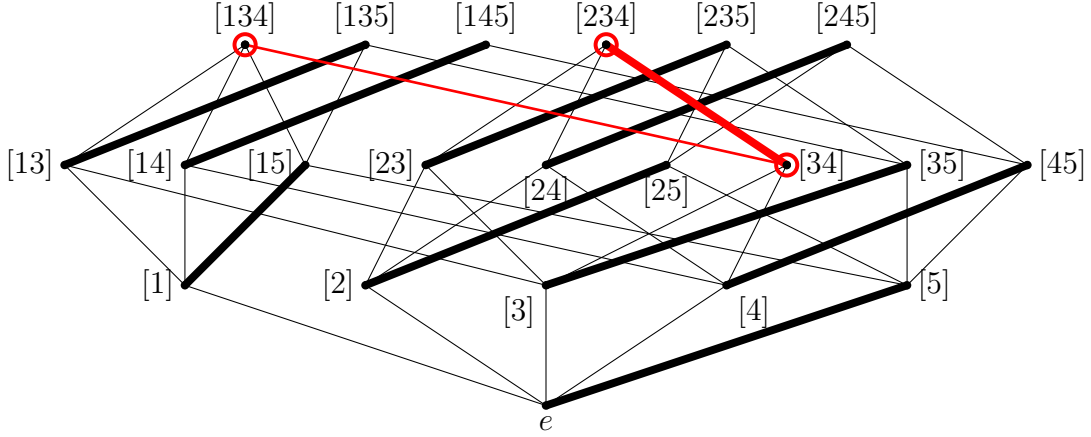
\end{enumerate}
\end{example}

\subsection{Optimal rank and optimal partner}\label{sec:matchings.4}

We are now ready to describe the optimal rank and an optimal partner for any boolean permutation.

\begin{theorem}\label{thm:optimal rank and partner for boolean}
Let $v$ be a boolean permutation, and let $\rw{s}$ be an optimal run word for $v$. Write $s$ as a concatenation
$$r_1 \cdots r_{\run(v)},$$
where the $r_i$ are runs. For each $i$, let $\rw{t_i}$ be the optimal partner for $\rw{r_i}$ as determined by Lemma~\ref{lem:runs}. Then $w:= \rw{t_1 \cdots t_{\run(v)}}$ is an optimal partner for $v$, and
$$\optimalrank(v) = \ell(v) - \run(v).$$ 
\end{theorem} 

\begin{proof}
With $v$ and $w$ as described, we have that $B(v) \cap B(w) = B(v) \setminus Q$, where $Q$ is the set of elements involving at least one full run $r_i$.

For each $i$, set $[a_i,a_i+b_i]:= \supp(\rw{r_i})$. We can now describe an almost perfect matching of $B(v) \cap B(w)$. Consider an element $z \in B(v) \cap B(w)$. We define the matching by examining how much of $[a_i,a_i+b_i]$ is contained in $\supp(z)$, starting with $i = 1$ and increasing $i$ as needed. If we reach an $i$ for which $b_i = 0$, we immediately increase $i$ because $t_i = \emptyset$ and $\rw{a_i} \not\in B(v) \cap B(w)$. So in the following outline, assume that each $b_i > 0$.

\begin{itemize}
\item Consider $z \in B(v) \cap B(w)$. If $a_1 \not\in \supp(z)$ and $[a_1+1,a_1+b_1] \not\subseteq \supp(z)$, then match $z \longleftrightarrow z \addletter \s_{a_1}$ similarly to the proof of Lemma~\ref{lem:runs}. 
\item The elements in $B(v) \cap B(w)$ that are not yet matched are exactly those that contain all of $[a_1+1,a_1+b_1]$ in their supports. These form a coideal, so we can proceed inductively.
Now let $z$ be such an element. If $a_2 \not\in \supp(z)$ and, additionally, we have
$[a_2+1,a_2+b_2] \not\subseteq \supp(z)$, then match $z \longleftrightarrow z \addletter \s_{a_2}$
similarly to the proof of Lemma~\ref{lem:runs}.
\item The elements in $B(v) \cap B(w)$ that are not yet matched are exactly those that contain all of $[a_1+1,a_1+b_1] \cup [a_2+1,a_2+b_2]$ in their supports. These form a coideal, so we can proceed inductively. Now let $z$ be such an element and repeat the process with $i = 3$.
\item And so on.
\end{itemize}

At the end of this process, after $i = \run(v)$, we have an almost perfect matching of $B(v) \cap B(w)$, and the only unmatched element is $u := \rw{\hat{r}_1 \cdots \hat{r}_{\run(v)}}$, where $\hat{r}_i$ is the run $i$ with the element $a_i$ removed. Note that $\ell(u) = \ell(v) - \run(v)$. By Proposition~\ref{prop:bounding optimal rank}, this completes the proof.
\end{proof}

It is illuminating to see Theorem~\ref{thm:optimal rank and partner for boolean} demonstrated in an example. The example is, perhaps, too big for drawing the full poset, but we can describe the key pieces.

\begin{example}
Let $v = 5123678(12)49(10)(11) = \rw{(11)43(10)5216798} \in \mf{S}_{12}$. Our first step is to find an optimal run word for $v$. There are several options for this, including the words $\rw{(11)(10)945678321}$ and $\rw{(11)(10)943215678}$. In particular, $\ell(v) = 11$ and $\run(v) = 3$, so the theorem predicts $\optimalrank(v) = 8$. Using $\rw{(11)(10)945678321}$, the theorem produces the optimal partner $w = \rw{(10)9(11)(10)567845672132}$, and the single unmatched element in the almost perfect matching of $B(v) \cap B(w)$ described in the proof of Theorem~\ref{thm:optimal rank and partner for boolean} is $\rw{(11)(10)567832}$, which does indeed have length $8$.
\end{example}

\section{The main result}\label{sec:tableaux}

\subsection{The minimal number of runs via the Robinson-Schensted 
correspondence}\label{sec:tableaux.1}

The {\em Robinson-Schensted correspondence} provides a bijection
$$
\mathbf{RS}:\mf{S}_n\to \coprod_{\lambda\vdash n}
\mathrm{SYT}_\lambda\times \mathrm{SYT}_\lambda
$$
between $\mf{S}_n$ and the set of pairs of standard Young tableaux of the same shape
(this shape is supposed to be a partition of $n$). We fix such a bijection given by 
Schensted's insertion algorithm, see \cite{schensted,sagan}. For $w\in W$,
we have $\mathbf{RS}(w)=(P,Q)$, where $P$ is the insertion tableau and
$Q$ is the recording tableau, see \cite{sagan} for details. We denote by
$\shape(w)$ the shape of $P$, by $\RowBT_i(w)$ the contents of its $i$th row, 
and $\lambda_i(w) := |\RowBT_i(w)|$.

Recall from Theorem~\ref{thm:boolean characterization} that boolean 
permutations avoid the pattern $321$. The main result of \cite{schensted}
therefore gives restrictions on their shapes.

\begin{corollary}\label{cor:boolean shape has 2 rows}
{\hspace{1mm}}

\begin{enumerate}\renewcommand{\labelenumi}{(\alph{enumi})}
\item For any permutation $w$, the number $\lambda_1(w)$  is the length of 
a longest increasing subsequence in $w$. 
\item If $w$ is boolean, then $\shape(w)$ has at most two rows.
\end{enumerate}
\end{corollary}

We start with the following observation.

\begin{lemma}\label{lem:run effect on top row}
Let $w, x \in \mf{S}_n$ be permutations such that $x = \rw{r}$ for a run $r$. Then $|\lambda_1(wx)-\lambda_1(w)|\leq 1$.
\end{lemma}

\begin{proof}
Let $w = a_1 \cdots a_n$ in one-line notation, and suppose that $x = \rw{p(p+1)\cdots (p+q)}$. Then the one-line notation of $wx$ is as follows, using {\color{red}red} to mark the part that changes when multiplying $w \mapsto wx$:
\begin{equation}
a_1 \cdots a_{p-1}{\color{red}a_{p+1}a_{p+2}\cdots a_{p+q+1}a_p}a_{p+q+2}\cdots a_n.
\end{equation}
This means that we only change the relative order of one element, $a_p$, compared to the elements in $\{a_{p+1},\ldots, a_{p+q+1}\}$, leaving all other relative orders intact. Therefore the length of the longest increasing subsequence can change by at most $1$. Similar arguments apply to the run $\rw{p(p-1)\cdots (p-q)}$.
\end{proof}

The previous lemma lets us relate the size $\lambda_1(v)$ to the number of runs $\run(v)$.

\begin{corollary}\label{cor-wm-72}
For any boolean $v\in \mf{S}_n$, we have $n-\lambda_1(v)\leq \run(v)$. In particular, $\lambda_2(v)\leq \run(v)$.
\end{corollary}

We are now ready to equate $\lambda_2(v)$ to a statistic we have already encountered, when $v$ is boolean.

\begin{theorem}\label{thm:2nd row counts runs}
For any boolean permutation $v$, we have
$$\lambda_2(v) = \run(v).$$
\end{theorem}

\begin{proof}
It is enough to prove the claim under the assumption that 
$\supp(v)$ contains all simple reflections, for otherwise 
we can write $v$ as a product of  two shorter commuting boolean
elements and use induction.

We induct on the length of $v$.
If $v$ contains just one simple reflection, the claim is obvious.
If $\ell(v) > 1$ then, up to taking the 
inverse of $v$, we may assume that $2$ appears to the left of 
$1$ in any reduced word for $v$. Let $k\geq 2$
be maximal with the property that $i+1$ appears to the left of
$i$ in any reduced word for $v$, for all $i<k$. Then 
$v=\rw{k(k-1)\cdots 21}v'$, where $\supp(v')=[k+1,n-1]$
and,  clearly, $\run(v)\leq \run(v')+1$.

The permutation $v'$ fixes all $i \le k$, and hence $12\cdots k$ belongs to any 
longest increasing subsequence in the one-line notation of $v'$. 
Multiplying $v'$ by $\rw{k(k-1)\dots 21}$ moves $1$ rightward past $2,3,\dots,k,v'(k+1)$ in the one-line notation. 
Since $1$ is the smallest element, this operation can only keep or reduce the length of an
increasing subsequence, compared to what had appeared in $v'$. In fact, because $k\geq 2$, this
produces the inversion $2>1$ in $v$ and hence necessarily makes 
all longest increasing subsequences in $v$ strictly shorter than what had been in $v'$. 
Specifically, removing the $1$ in a longest increasing subsequence for 
$v'$ produces an increasing subsequence for $v$ which is shorter by exactly one term.
It follows that $\row_2(v)=\row_2(v')+1$. Combining this with
Corollary~\ref{cor-wm-72} and $\run(v)\leq \run(v')+1$, by 
induction we have $\run(v)=\run(v')+1$. This implies $\row_2(v) = \run(v)$,  proving the claim.\end{proof}

\subsection{Lusztig's $\mathbf{a}$-function for the symmetric group}\label{sec:final.1}

In \cite{lusztig}, Lusztig introduced the function $\mathbf{a}:W\to\mathbb{Z}_{\geq 0}$,
where $W$ is a Coxeter group, with the following properties:
\begin{itemize}
\item $\mathbf{a}$ is constant on two-sided Kazhdan-Lusztig cells in $W$;
\item $\mathbf{a}(w)\leq \ell(w)$, for all $w\in W$;
\item $\mathbf{a}(w)=\ell(w)$ if $w$ is the longest element of some
parabolic subgroup of $W$.
\end{itemize}
In the special case of a symmetric groups, it is well-known, see \cite{KL},
that two permutations $v$ and $w$ belong to the same two-sided Kazhdan-Lusztig
cell if and only if $\mathsf{sh}(v)=\mathsf{sh}(w)$.

Given a partition $\lambda\vdash n$, consider the transposed partition
$\mu:=\lambda^t=(\mu_1,\mu_2,\dots,\mu_m)$. Consider the parabolic subgroup 
$W_\mu$ of $\mf{S}_n$ given by $\mf{S}_{\mu_1}\times \mf{S}_{\mu_2}\times \cdots
\times \mf{S}_{\mu_m}$.
Then it is easy to check that the Robinson-Schensted correspondent of
the longest element in $W_\mu$ has shape $\lambda$. In particular, 
each two-sided Kazhdan-Lusztig cell contains the longest element of some parabolic 
subgroup. Therefore the properties of  $\mathbf{a}$ listed above 
determine the function $\mathbf{a}$ for $\mf{S}_n$ uniquely.

\begin{lemma}\label{lem-wm-51}
Let $w\in \mf{S}_n$ be such that $\mathsf{sh}(w)=\lambda$. 
Then, for $\mu=\lambda^t$, we have
$$\displaystyle\mathbf{a}(w)=\sum_{i=1}^m\frac{\mu_i(\mu_i-1)}{2}.$$ 
\end{lemma}

\begin{proof}
This follows directly from the above and the
fact that the length of the longest element in $\mf{S}_k$
equals $\frac{k(k-1)}{2}$, for any $k$.
\end{proof}

\begin{corollary}\label{lem-wm-52}
If $v\in \mf{S}_n$ is boolean and $\mathsf{sh}(v)=\lambda$, 
then $\mathbf{a}(v)=\lambda_2(v)$.
\end{corollary}

\begin{proof}
If $v$ is boolean, $\lambda$ has at most two rows.
Therefore  $\mu=\lambda^t=(2^{\lambda_2},1^{n-2\lambda_2})$.
Now the claim follows directly from Lemma~\ref{lem-wm-51}.
\end{proof}

\subsection{Grades of simple modules for boolean elements via Lusztig's $\mathbf{a}$-function}
\label{sec:final.3}

We can now prove our main result.

\begin{theorem}\label{thmmain}
Let $v\in\mf{S}_n$ be boolean. Then $\mathbf{grade}(L_v)=\mathbf{a}(v)$.
\end{theorem}

\begin{proof}
After the discussion in Section~\ref{sec:motivation}, the claim 
follows  by combining Theorems~\ref{thm:optimal rank and partner for boolean} and \ref{thm:2nd row counts runs}
with Corollary~\ref{lem-wm-52}.
\end{proof}

We note that Lusztig's $\mathbf{a}$-function describes various homological
invariants in BGG category $\mathcal{O}$, see \cite{M1,M2,KMM}.

\section{Grades of simple modules for non-boolean elements}\label{sec:nonboolean}

\subsection{Longest elements in parabolic subgroups}\label{sec:nonboolean.1}

We conclude by remarking upon how this work does and does not extend to non-boolean elements. 
Rene Marczinzik has computed the grades of all simple modules 
for $\mf{S}_4$ (over $\mathbb{C}$) using a computer. In that case it turns out that 
$\mathbf{grade}(L_w)\neq \mathbf{a}(w)$, for the (non-boolean) permutations $w=\rw{2132}$ and $w=\rw{12321}$.
This means that Theorem~\ref{thmmain} does not generalize to all elements of $W$. It does, however, hold true in another special case, in some sense, the ``opposite extreme'' of the boolean elements.

\begin{theorem}\label{thm-main-2}
Let $v\in\mf{S}_n$ be the longest element of some parabolic subgroup. 
Then $\mathbf{grade}(L_v)=\mathbf{a}(v)=\ell(v)$. 
\end{theorem}

\begin{proof}
If $v=e$, then the claim is obvious. Therefore we assume $v\neq e$.
Let $w\in \mf{S}_n$. Consider some reduced word $\rw{s} \in R(w)$.
Let $\rw{x}$ be the shortest prefix of $\rw{s}$ with the property that 
no simple reflection in the suffix of $\rw{s}$, defined as the complement 
to $\rw{x}$, belongs to the support of $v$.  From the subword property
(see Theorem~\ref{thm:subword property}), it follows
that $B(v)\cap B(\rw{x})=B(v)\cap B(w)$. We have to consider two cases.

{\bf Case~1: $\rw{x}=e$.} In this case $B(v)\cap B(\rw{x})=\{e\}$ and 
hence the corresponding complex \eqref{eq-wm-5}  
is concentrated in one degree, namely, in degree $-\ell(v)$.

{\bf Case~2: $\rw{x}\neq e$.} In this case, due to the minimality of $\rw{x}$,
the rightmost letter of $x$ belongs to the support of $v$. Therefore the
right descent set of $\rw{x}$ contains a simple reflection that belongs to
the support of $v$. Since $v$ is the longest element in some parabolic
subgroup, its support coincides with both its left descent set and
its right descent set. In particular, $\rw{x}$ and $v$ have a common 
simple reflection in the right descent set. As explained in 
Subsection~\ref{sec:motivation.7}, this implies exactness of \eqref{eq-wm-5}.
Consequently, this case does not effect the computation of $\mathbf{grade}(L_v)$.

It follows that $\mathbf{grade}(L_v)=\ell(v)$ and the claim of the
theorem now follows from the property $\mathbf{a}(v)=\ell(v)$,
for $v$ the longest element of a parabolic subgroup.
\end{proof}

\begin{remark}
One could observe that the permutations $w=\rw{2132}$ and $w=\rw{12321}$ are exactly 
the two elements of $\mf{S}_4$ for which the corresponding Kazhdan-Lusztig polynomial 
$P_{e,w}$ is non-trivial. By a result of Deodhar, see \cite{De},  this 
condition is equivalent to nonsingularity of the Schubert variety for $w$.

Unfortunately, at the present stage we do not know whether 
it is reasonable to extrapolate this observation to a guess for higher ranks.
In order to investigate this kind of guess, we need a better understanding of 
the combinatorial structure of intersections of principal Bruhat ideals.
Maybe the recent preprints \cite{BBDVW,tenner bruhat intersections}, which appeared after the
preprint version of the present paper, will be helpful.
\end{remark}

\subsection{Classification of perfect simple module}\label{sec:nonboolean.2}

Recall that a module is called {\em perfect} if its grade coincides with
its projective dimension. Theorem~\ref{thm-main-2} leads to the following
classification of perfect simple modules.

\begin{theorem}\label{thm-main-3}
For  $w\in\mf{S}_n$, the module $L_w$ is perfect if and only if $w$ is the
longest element in some parabolic sugroup of $\mf{S}_n$.
\end{theorem}

\begin{proof}
By Proposition~\ref{propwm-2}, the projective dimension of $L_w$ is given by
$\ell(w)$. This means that the ``if'' part of the claim
follows directly from Theorem~\ref{thm-main-2}.

To prove the ``only if'' part, suppose that $w$ is not the longest element in any 
parabolic subgroup. Let $G$ be the minimal parabolic subgroup containing $w$.
Let $u \in G$ be minimal having the property that $u \not\in B(w)$.
Because $w$ is not the maximum element of $G$, such  $u$  exists.
Since $G$ is the minimal parabolic subgroup containing $w$, 
all simple reflections generating $G$ belong to the support of $w$,
in particular, they are all less than $w$ in the Bruhat order and hence cannot coincide
with $u$. This implies that $\ell(u) > 1$.

Consider $B(w) \cap B(u)$. By construction, we have $B(w) \cap B(u) = B(u) \setminus \{u\}$.
We know that there is a perfect matching, call it $F$, of the elements of $B(u)$. 
Restricting $F$ to $B(u)\setminus\{u\}$ gives an almost perfect matching whose only 
unmatched element is $u$'s partner under $F$. This unmatched element of $B(w) \cap B(u)$ has length $\ell(u) - 1$.
Therefore the grade of $w$ is at most $\ell(w) - (\ell(u) - 1) < \ell(w)$
which means that $L_w$ is not perfect.
\end{proof}

We note that both Theorem~\ref{thm-main-2} and \ref{thm-main-3} are true, 
with the same proofs, for arbitrary finite Weyl groups.

\section*{Acknowledgments}
We thank Rene Marczinzik both for his talk that started our interest in the problem discussed in the paper
and for very useful comments on the preliminary version. 
We also thank Axel Hultman for very helpful discussions. 
Finally, we are grateful for the careful reading and advice of the anonymous referees.

\end{document}